\documentclass[12pt]{amsart}

\usepackage{amssymb,amsmath}
\usepackage{amsthm}
\usepackage{tikz}
\usepackage{fancyhdr}
\usepackage[matrix,arrow,tips]{xy}
\usepackage{fullpage}
\usepackage{booktabs}
\usepackage{enumerate}
\usepackage{tabu}

\usepackage{overpic}

\numberwithin{equation}{section}
\newcommand{\eps}{\epsilon}

\newcommand{\bbC}{\mathbb{C}}
\newcommand{\bbD}{\mathbb{D}}

\newcommand{\bbN}{\mathbb{N}}
\newcommand{\bbR}{\mathbb{R}}
\newcommand{\bbP}{\mathbb{P}}

\newcommand{\calD}{\mathcal{D}}
\newcommand{\calE}{\mathcal{E}}
\newcommand{\calF}{\mathcal{F}}

\newcommand{\calJ}{\mathcal{J}}

\newcommand{\del}{\partial}

\newcommand{\tr}{\mathrm{tr}}
\newcommand{\End}{\mathrm{End}}

\newcommand{\Hom}{\mathrm{Hom}}

\newcommand{\diag}{\mathrm{diag}}

\newcommand{\laa}{\mathfrak{a}}

\newcommand{\lag}{\mathfrak{g}}
\newcommand{\lah}{\mathfrak{h}}

\newcommand{\lak}{\mathfrak{k}}
\newcommand{\lam}{\mathfrak{m}}

\newcommand{\lap}{\mathfrak{p}}

\newcommand{\lat}{\mathfrak{t}}
\newcommand{\lau}{\mathfrak{u}}

\newcommand{\lasl}{\mathfrak{sl}}

\newcommand{\GL}{\mathrm{GL}}
\newcommand{\SL}{\mathrm{SL}}

\newcommand{\SU}{\mathrm{SU}}
\newcommand{\U}{\mathrm{U}}

\newcommand{\sph}{\mathrm{sph}}

\newcommand{\ind}{\mathrm{ind}}

\newcommand{\reg}{\mathrm{reg}}

\newcommand{\Lie}{\mathrm{Lie}}

\newcommand{\pol}{\mathrm{pol}}
\newcommand{\rad}{\mathrm{rad}}
\newcommand{\ls}{\lambda_\mathrm{sph}}

\renewenvironment{proof}{\noindent{\scshape Proof.}}{\qed}

\theoremstyle{plain}
\newtheorem{theorem}{Theorem}[section]
\newtheorem{lemma}[theorem]{Lemma}
\newtheorem{proposition}[theorem]{Proposition}
\newtheorem{corollary}[theorem]{Corollary}

\newtheorem{definition}[theorem]{Definition}

\newtheorem{condition}[theorem]{Condition}

\theoremstyle{definition}

\newtheorem{remark}[theorem]{Remark}

\title[Matrix elements and multivariable matrix-valued orthogonal polynomials]
{Matrix elements of irreducible representations of $\SU(n+1)\times\SU(n+1)$ and multivariable matrix-valued orthogonal polynomials}

\author{Erik Koelink}
\address{IMAPP, Radboud Universiteit, 
PO Box 9010, 6500 GL Nijmegen, 
the Netherlands}
\email{e.koelink@math.ru.nl}
\author{Maarten van Pruijssen}
\address{Institut f\"ur Mathematik,
Universit\"at Paderborn,
Warburger Str. 100,
D-33098 Paderborn, Germany}
\email{vanpruijssen@math.uni-paderborn.de} 
\author{Pablo Rom{\'a}n}
\address{CIEM, FaMAF, Universidad Nacional de C\'ordoba, Medina Allende s/n Ciudad
Universitaria, C\'ordoba, Argentina}
\email{roman@famaf.unc.edu.ar}

\date{\today}

\begin{document}

\begin{abstract}
In Part 1 we study the spherical functions on compact symmetric pairs of arbitrary rank under a suitable multiplicity freeness assumption 
and additional conditions on the branching rules. The spherical functions are taking values in the spaces of linear operators of 
a finite dimensional representation of the subgroup, so the spherical functions are matrix-valued. 
Under these assumptions these functions can be described in terms of matrix-valued orthogonal polynomials in several variables,
where the number of variables is the rank of the compact symmetric pair.
Moreover, these polynomials are uniquely determined as simultaneous eigenfunctions of a commutative algebra of differential operators.

In Part 2 we verify that the group case $\SU(n+1)$ meets all the conditions that we impose in Part 1. 
For any $k\in\bbN_{0}$ we obtain families of orthogonal polynomials in $n$ variables with values in the $N\times N$-matrices, where $N=\binom{n+k}{k}$.
The case $k=0$ leads to the classical Heckman-Opdam polynomials of type $A_{n}$ with geometric parameter.
For $k=1$ we obtain the most complete results.
In this case we give an explicit expression of the matrix weight, which we show to be irreducible whenever $n\ge2$.
We also give explicit expressions of the spherical functions that determine the matrix weight for $k=1$.
These expressions are used to calculate the spherical functions that determine the matrix weight for general $k$ up to invertible upper-triangular matrices.
This generalizes and gives a new proof of a formula originally obtained by Koornwinder for the case $n=1$.
The commuting family of differential operators that have the matrix-valued polynomials as simultaneous eigenfunctions contains an element of order one.
We give explicit formulas for differential operators of order one and two for $(n,k)$ equal to $(2,1)$ and $(3,1)$.
\end{abstract}

\maketitle

\tableofcontents



\section{Introduction}

\subsection{Motivation and history}
There is an intimate relationship between special functions and group theory. 
It consists of a very fruitful cross-fertilization which has been 
exploited in several directions. 
Typically, matrix coefficients of compact or complex groups are related to polynomials in various forms. 
In this paper, we explore this relationship further and we discuss multivariable matrix-valued 
orthogonal polynomials related to the representation theory of compact groups.
The relation is established by exploiting properties of matrix-valued spherical functions. 
Matrix-valued spherical functions extend the notion of zonal spherical functions on symmetric spaces.
They have been studied extensively by Harish-Chandra, see e.g.~\cite{CM1982, Warner2} for an account, and subsequently by several  other authors to understand the harmonic analysis on real reductive groups, see e.g.~\cite{Camporesi1997, GangVara, Gode, Pedon, tirao, Warner2}.
The successful relation between harmonic analysis on compact symmetric spaces and orthogonal polynomials via the 
study of spherical functions, of which the spherical harmonics on the sphere are a prototype, has been 
described and studied in e.g.~\cite{vanDijk,Heckman,HelgasonGGAnew,VileK}. 

Matrix-valued orthogonal polynomials of a single variable have been introduced in the 1940s by M.G.~Krein 
in the study of operators with higher order deficiency indices. 
Krein also studied the corresponding moment problem in the 
context of spectral theory. The study of matrix-valued orthogonal polynomials 
has several applications, see the overview paper \cite{DamaPS}
for an introduction and references up to 2008. 
One of the developments in the study of matrix-valued orthogonal polynomials
is extending classical results for scalar-valued orthogonal polynomials to the setting of 
matrix-valued orthogonal polynomials.
This includes the study of the matrix-valued differential 
operators having these matrix-valued orthogonal polynomials as eigenfunctions,
which in general leads to a non-commutative algebra of differential operators. 
The construction of interesting examples of matrix-valued orthogonal polynomials that are simultaneous eigenfunctions of matrix-valued differential operators
had been lagging behind until 2002.

The first paper establishing explicit classes of matrix-valued orthogonal polynomials using 
matrix-valued spherical functions and differential operators was the paper \cite{GPT} by Gr\"unbaum, Pacharoni, Tirao. 
In this paper matrix-valued spherical functions for the 
compact symmetric pair $(\SU(3),\U(2))$ were
considered. The approach relies on the reduction of such a 
matrix-valued spherical function to a matrix-valued function on 
the corresponding symmetric space $\mathbb{P}^2(\bbC)= \SU(3)/\U(2)$ 
and heavy usage of matrix-valued differential operators which are known explicitly for this case. 
The approach of \cite{GPT} turns out to be too complicated in general to generalize to pairs of compact groups where there is less control over the differential operators.

Motivated by Koornwinder's paper \cite{Koor1} on vector-valued 
orthogonal polynomials, we have developed an approach for matrix-valued orthogonal polynomials
for the compact symmetric space $(G,K)=(\SU(2)\times \SU(2),\diag\,\SU(2))$ in which all the main properties are explicit \cite{KvPR,KvPR2}.
These main properties include the 
orthogonality relations, in particular two explicit descriptions of the matrix-valued weight, 
the three-term recurrence relation, explicit description of the reducibility, 
two explicit commuting matrix-valued differential operators having the matrix-valued orthogonal polynomials
as eigenfunctions, the explicit relationship 
to Tirao's matrix-valued hypergeometric functions, etc. All results in the papers \cite{KvPR, KvPR2} 
are obtained for arbitrary dimensions of the matrix algebras.

The study of this example has led to a general theory for the matrix-valued orthogonal
polynomials in relation to Gelfand pairs of rank one, see 
\cite{HvP,MVPthesis,vPR}.
To set up the general theory we have to impose multiplicity-free restriction in the branching rules for certain representations of the groups that are involved. 
Then the group theoretic interpretation gives a commutative class of matrix-valued differential operators to which 
these matrix-valued orthogonal polynomials are eigenfunctions. 
These differential operators arise naturally from a suitable subalgebra of the universal enveloping algebra,
which includes the Casimir element \cite{Dixmier}. To obtain them we have to perform radial part calculations, see \cite{CM1982}, 
and conjugations with suitable matrix-valued functions. 
The general set-up from \cite{HvP,MVPthesis} also applies to the examples calculated in \cite{GPT,TZ} where matrix-valued orthogonal polynomials are obtained from studying the differential equations.

\subsection{Results}
One of the main results in \cite{mvp-MVMVOP} is the existence of families of multivariable matrix-valued orthogonal polynomials that are simultaneous 
eigenfunctions of a commutative algebra of differential operators. The existence is based on examples and an \textit{ad hoc} analysis of the involved spectra. 
In this paper we present a solid theory for the general construction of the polynomials and the differential operators based on three isolated conditions. 
These conditions are satisfied by the pairs $(\SU(n+1)\times\SU(n+1),\diag\,\SU(n+1))$ and the irreducible representations of $\SU(n+1)$ on $S^{k}(\bbC^{n+1})$, 
the $k$-th symmetric power of the standard representation. For this class of examples we are able to provide many explicit expressions. 
In particular for $k=1$ we give an explicit formula of the weight-matrix and prove its irreducibility. We also provide explicit expressions of commuting differential operators in low dimensions. We proceed with a detailed discussion of our results.

In  Part \ref{part:generalities} of the paper 
we set up a general theory on the relationship between 
multivariable matrix-valued orthogonal polynomials and the representation theory of a compact symmetric space $U/K$. 
First we study matrix-valued spherical functions in some detail.
Fixing a $K$-representation $\pi_\mu^K$ of highest weight $\mu$ in the space 
$V_\mu^K$, we study the space $E^{\mu}$ of matrix-valued functions $\Phi^\mu$ on $U$ taking values in $\End(V_\mu^K)$ so that 
\[
\Phi^\mu(k_1gk_2) = \pi_\mu^K(k_1) \Phi^\mu(g) \pi_\mu^K(k_2), \qquad 
\forall\, k_1,k_2\in K, \ \forall\, g\in G. 
\]
We look for $U$-representations of highest weight $\lambda$ so that we 
can associate a non-trivial matrix-valued spherical function $\Phi^{\mu}_{\lambda}$, see \eqref{def: spherical function}, to this 
representation. These are the irreducible representations of $U$ whose restriction to $K$ contains $\pi^{K}_{\mu}$.
The highest weights of these representations are collected in the set $P^{+}_{U}(\mu)$.
The first condition that we impose is multiplicity freeness: we fix an irreducible representation $\pi^{K}_{\mu}$ such that $[\pi^{U}_{\lambda}|_{K}:\pi^{K}_{\mu}]=1$ for all $\lambda\in P^{+}_{U}(\mu)$.

For example, take $\pi_\mu^K$ the trivial representation, i.e.~$\mu=0$. 
The first condition is satisfied by the Cartan-Helgason Theorem \cite[Thm.~8.49]{KnappLGBI}. 
By the same theorem, the set $P^{+}_{U}(0)$ is a semi-group generated by $n$ elements $\lambda_{1},\ldots,\lambda_{n}$, where $n$ is the rank of the symmetric space.
The space $E^{0}$ of $K$-biinvariant functions is generated by 
fundamental zonal spherical functions $\phi_{1},\ldots,\phi_{n}$, i.e.~the spherical functions of type $\pi^{K}_{0}$ related to the fundamental spherical weights $\lambda_{1},\ldots,\lambda_{n}$. For the general case we impose the following condition on $P^{+}_{U}(\mu)$, namely that it is of the form
$$P^{+}_{U}(\mu)=B(\mu)+P^{+}_{U}(0),$$
where $B(\mu)$ is a finite subset of dominant integral weights. This condition is satisfied for $\mu=0$ by taking $B(0)=\{0\}$.

The set $B(\mu)=\{\nu_{1},\ldots,\nu_{N}\}$ provides $N$ ``minimal spherical functions $\Phi^{\mu}_{\nu_{i}}$ of type $\pi^{K}_{\mu}$''. 
Our third condition, which is of a technical nature, ensures that we can write an element $\Phi^{\mu}\in E^{\mu}$ as an 
$E^{0}$-linear combination of the minimal spherical functions of type $\pi^{K}_{\mu}$, i.e.~there exist polynomials $q(\Phi^{\mu},i)\in\bbC[\phi_{1},\ldots,\phi_{n}]$ such that
$$\Phi^{\mu}=\sum_{i=1}^{N}q(\Phi^{\mu},i)(\phi_{1},\ldots,\phi_{n})\Phi^{\mu}_{\nu_{i}}.$$
This construction then allows us to define the multivariable matrix-valued orthogonal polynomials by collecting the polynomials in the fundamental zonal spherical functions in a systematic way.

The matrix-valued orthogonality measure can be given explicitly.
The orthogonality measure involves a matrix part which 
involves the matrix-valued spherical functions associated to the set $B(\mu)$. 
We take this information together in a matrix-valued function $\Psi^{\mu}_0$, and then
the matrix part of the orthogonality measure is given by $(\Psi^{\mu}_0)^\ast T^{\mu}\Psi^{\mu}_0$, where $T^{\mu}$ is a diagonal matrix whose entries depend on the elements in $B(\mu)$. 
In particular, the size $N$ of the algebra of $N\times N$-matrices in which these 
polynomials take their values equals $\#B(\mu)$.
The orthogonality measure also involves a scalar part and this part 
requires the knowledge of the decomposition of the Haar measure with respect to 
the $KAK$-decomposition.

In order to obtain the matrix-valued differential operators for the 
multivariable matrix-valued orthogonal polynomials we need to
perform radial part calculations to find the matrix-valued 
differential operators for the matrix-valued spherical functions,
following \cite{CM1982}. Next we need to conjugate 
these operators with the matrix-valued function $\Psi_0^{\mu}$ 
to come to a result for the matrix-valued polynomials, and this 
requires matrix-valued differential equations for $\Psi_0^{\mu}$ of order lower than the order of the initial differential operator. 
Finally, we need to switch to coordinates in terms of the fundamental zonal 
spherical functions and finally to real coordinates. 

In the second part of the paper, we make this program explicit for the case 
of the symmetric space $(U,K)=(\SU(n+1)\times \SU(n+1),\SU(n+1))$,
where $\SU(n+1)$ is diagonally embedded as the fixed point set of the 
flip.
Part \ref{part:specialcase} extends the case 
$n=1$ studied previously in \cite{KvPR,KvPR2}. 
We show that the conditions on inverting the 
branching rules is satisfied in case we take the $\SU(n+1)$-representations $S^k(\bbC^{n+1})$ of highest weight $\mu=k\omega_1$. 
The branching rules are described using the theory of spherical varieties in Section \ref{sec:invertingbranchingrule}
and we show that in these cases all conditions of the general part are satisfied.
The zonal spherical functions generating $K$-biinvariant functions are the characters.
The explicit orthogonality relations involve the Dyson integral --a special case of the 
Selberg integral-- as well as the determination of some explicit constants.
We show that the matrix-valued weight is irreducible for $n\geq 2$ and $k=1$.
It is known that this is not the case for $n=1$, see \cite{KvPR, KvPR2}.
The orthogonality measure is described in terms of the matrix-valued spherical functions 
corresponding to the representations labeled by the weights in the set $B(\mu)=B(k\omega_1)$, which 
we collect in a matrix-valued function $\Phi_0$.
The most elementary case $k=1$ of $\Phi_0$ gives a $(n+1)\times (n+1)$-matrix 
which can be viewed as a kind of group element $g_a$, parametrized by $a\in A_{c}$, where $A_{c}$ is the 
compact torus of the $U=KA_{c}K$-decomposition.
We show that for the more general cases, i.e. for  $k>1$, the corresponding matrix-valued function $\Phi_0$ can be obtained in terms of a suitable representation evaluated at $g_a$ up to constant matrices. 
This result is inspired by the remarkable observation of Koornwinder for the case $n=1$ in \cite[Prop.~3.2]{Koor1}. 
The proof that we present is of a different nature, hence we obtain a new proof of Koornwinder's result.
The generalization of Koornwinder's result implies that the case $k=1$ is fundamental to understand $\Phi_0$ for 
arbitrary $k\in \bbN_{0}$, which in turn is essential to find the matrix part of the weight. 
The scalar part of the orthogonality measure is supported on the interior of a compact set in $\bbR^n$ after a change of coordinates.
For $n=1$ it is supported on the interval $[-1,1]$, for $n=2$ on the interior of Steiner's hypocycloid and for $n=3$ on a 
$3$-dimensional analog of Steiner's hypocycloid, see Figure \ref{fig:2d}. 

Using Dixmier \cite{Dixmier}, we find a commutative subquotient $\bbD(\mu)$ of the universal enveloping algebra whose elements act as differential operators 
having the matrix-valued spherical functions as eigenfunctions, see also \cite{Deitmar90}. 
In particular, this symmetric space comes naturally with two Casimir operators, one from the first factor
of $\SU(n+1)\times\SU(n+1)$ and one form the second. Using radial part calculations \cite{CM1982}, this leads to two second order matrix-valued 
differential operators for the associated multivariable matrix-valued orthogonal 
polynomials after conjugation with $\Phi_0$ and 
a change of coordinates. The difference of these two operators leads to a first(!)-order  matrix-valued differential
operator having the matrix-valued orthogonal polynomials as simultaneous eigenfunctions. 
This is remarkable, since for scalar-valued orthogonal polynomials this is not possible by Bochner's Theorem.
For single-variable matrix-valued or multivariable scalar-valued
orthogonal polynomials there is no known example of this phenomenon, see for instance the discussion in
\cite[p.155]{DdlI} and references therein. 
We present some of these operators in explicit low-dimensional cases, for $n=2,3$ and $k=1$. 
The explicit case in Part \ref{part:specialcase} in the scalar case for $n=2$ reduces to 
the $2$-variable orthogonal polynomials on (the interior of) Steiner's hypocycloid, see Figure \ref{fig:2d}, introduced 
by Koornwinder \cite{KoorIM} in the 1970s. So for $n=2$ we have constructed $2$-variable matrix-valued 
analogues of Koornwinder's orthogonal polynomials on Steiner's hypocycloid. The dimension $N$ of 
the $N\times N$-matrix-valued orthogonal polynomials is $\#B(k\omega_1)=\dim_\bbC (\text{End}_M (S^k(\bbC^3)))=\dim_{\bbC}(S^{k}(\bbC^{3}))$, which is 
$N=\frac12(k+2)(k+1)$. Here $M=Z_K(A_{c})$, which is a maximal torus in $K$. 

The results are written in terms of polynomials in the zonal spherical functions where the degree is a multi-index. 
The Heckman-Opdam polynomials of type $A$ (for the geometric parameter) are written as symmetric functions in the 
coordinates on the abelian subgroup $A_c$ and indexed by partitions. In the scalar case, the correspondence is 
given by the coordinate transformation which rewrites a symmetric polynomial as a polynomial in the elementary symmetric functions.
The reason to write it in this way is that the general construction in Part \ref{part:generalities} gives the results 
naturally in terms of zonal spherical functions times matrix-valued spherical functions corresponding to 
minimal representations of $B(\mu)$. We obtain symmetric functions only at a later stage, e.g. after writing
down the orthogonality relations explicitly. 

\subsection{Outlook}
It is well-known that Koornwinder's original 1970s papers have been very influential in the 
development of the multivariable Heckman-Opdam polynomials and functions, which in turn play an 
important role in integrability of systems such as the Calogero-Moser-Sutherland models, see \cite{Heckman}.
A natural question is whether or not there is an extension of Cherednik's approach 
or an application of Dunkl operators available for these multivariable matrix-valued polynomials,
see \cite{DunkX, Macd2, OpdamLectureNotesDunkl}.  
The possible application to integrable systems of the class of polynomials as in this paper remains to be
investigated. 
Also, it might be possible to extend some of the results of this paper to more general
parameters, which has been done for $n=1$ of Part \ref{part:specialcase} in \cite{KdlRR}.
Similarly, one may consider the extension to the quantum setting 
and to obtain quantum analogues of the polynomials of this paper, see 
\cite{AldeKR} for the 
quantum analogue of the case $n=1$ of Part \ref{part:specialcase}.

One can also consider the spherical functions of fixed $K$-type on non-compact symmetric spaces. In this case we expect multivariable matrix-valued special functions that are eigenfunctions to the same algebra of differential operators.
However, the set of parameters needs to be enlarged and requires further study, e.g.~because of the possible occurrence of discrete series representations. Certain properties of the eigenfunctions, such as asymptotic behavior, 
were already understood by Harish-Chandra, see e.g.~\cite{CM1982} for an account. 
Some references that consider these questions are \cite{Camporesi2000, RomanTirao}.

\subsection*{Acknowledgement}
This research was supported through the program `Research in Pairs' by the Mathematisches Forsungsinstitut 
Oberwolfach in 2017. The research of Pablo Rom\'an was partly done while being a Radboud Excellence Initiative Fellow
at Radboud Universiteit. 
Erik Koelink gratefully acknowledges the support of FaMAF as invited professor at Universidad Nacional de C\'ordoba 
for a research visit. 
Pablo Rom\'an was supported by FONCyT grant PICT 2014-3452 and by SECyT-UNC.


\part{Generalities on spherical functions in the multiplicity free setting}\label{part:generalities}


\section{Matrix-valued spherical functions}\label{sec:MVSFs}

We recall some of the results of \cite{mvp-MVMVOP} specified to symmetric spaces. The main idea, to view the space spanned by the spherical functions as a module over the ring of biinvariant functions, originates from \cite{HvP,MVPthesis} and is based on the classical results in \cite{Vretare1976}. 
Let $(G,H)$ be a complex symmetric pair of rank $n$, and let $(U,K)$ be the corresponding 
compact symmetric pair.  
We assume that $G$ is connected and semisimple and that $H$ is connected.  

We let $H\times H$ act on the regular functions $\bbC[G]$ by the biregular representation given by $(h_{1},h_{2})f(g)=f(h_{1}^{-1}gh_{2})$. 
Let $E^{0}=\bbC[G]^{H\times H}$ denote the algebra of $H$-biinvariant regular functions. 
Suppose we have chosen a Borel subgroup of $H$, and that we are given an irreducible representation $(\pi^{H}_{\mu},V^{H}_{\mu})$ of $H$, where $\mu$ is the highest weight according to the choice of the Borel subgroup. The (finite dimensional) vector space is also called an $H$-module of highest weight $\mu$ and we sometimes simply write $V$ or $V_{\mu}$ instead.

The corresponding representation of $K$ in $V$ is unitary for a fixed inner product, which
we assume is anti-linear in the first leg. By Weyl's unitary trick we identify the representations of $K$ and $H$ and
the representations of $U$ and $G$.

The group $H\times H$ acts 
naturally on $\bbC[G]\otimes\End(V)$ by the biregular representation in the 
first leg of the tensor product and by left multiplication by $\pi^{H}_{\mu}(h_1)$ 
and by right multiplication by $\pi^{H}_{\mu}(h_2^{-1})$ in the second leg.
The space of invariants $E^{\mu}=(\bbC[G]\otimes\End(V))^{H\times H}$ is the space
of $\End(V)$-valued holomorphic polynomials on $G$ satisfying 
\begin{equation}\label{eqn: trafo}
F(h_{1}gh_{2})=\pi^{H}_{\mu}(h_{1})F(g)\pi^{H}_{\mu}(h_{2}),\quad\mbox{for all $h_{1},h_{2}\in H$ and $g\in G$}.
\end{equation}
Note that the trivial representation $\mu=0$ gives back the space 
$E^{0}$ of $H$-biinvariant holomorphic polynomials.
Note that the space of invariants $E^{\mu}$ is a $E^{0}$-module by point-wise multiplication.  

To analyze $E^{\mu}$ we use explicit knowledge of the decomposition of the $G$-module $\ind^{G}_{H}\pi^{H}_{\mu}$. 
We collect the highest weights (after having fixed a Borel subgroup $B_{G}\subset G$ and a maximal torus $T_{G}\subset B_{G}$) of the irreducible $G$-subrepresentations of 
$\ind^{G}_{H}\pi^{H}_{\mu}$ in the set
\begin{eqnarray}\label{eq: mu-well}
P^{+}_{G}(\mu)=\{\lambda\in X^{+}(T_{G})\,|\, [\pi^{G}_{\lambda}|_{H}:\pi^{H}_{\mu}]\ge1\}.
\end{eqnarray}
In order to further analyze the space $E^{\mu}$ and to establish a connection with matrix-valued orthogonal polynomials we impose conditions on the data $(G,H,\mu)$. The first condition is
on the set $P^{+}_{G}(\mu)$.

\begin{condition}\label{cond:multfree} $(G,H,\mu)$ is a multiplicity free triple, i.e.~$\ind^{G}_{H}\pi^{H}_{\mu}$ decomposes multiplicity free. 
\end{condition}


There is an abundance of examples of multiplicity free triples, namely those coming from the multiplicity free systems, i.e. triples $(G,H,P)$ with $(G,H)$ as before and with $P\subset H$ a parabolic subgroup such that $G/P$ admits an open orbit of a Borel subgroup of $G$. 
Any positive character $\mu\in X^{+}(T_{H})$ that extends to a character $\mu\colon P\to\bbC^{\times}$ gives rise to a multiplicity free triple, see \cite{mvp-MVMVOP}.

A spherical function of type $\mu$, associated to $\lambda\in P^{+}_{G}(\mu)$ with 
$G$-representation $\pi^G_{\lambda}$ acting in $V^{G}_{\lambda}$ is defined as  
\begin{eqnarray}\label{def: spherical function}
\Phi^{\mu}_{\lambda}\colon G\to\End(V^{H}_{\mu}):g\mapsto p\circ\pi^{G}_{\lambda}(g)\circ j,
\end{eqnarray} 
where $j\colon V_{\mu}\to V^{G}_{\lambda}$ is an $H$-equivariant embedding, unitary for the $U$ and 
$K$-invariant inner products on the respective representation spaces $V^{H}_{\mu}$ and $V^{G}_{\lambda}$.
The map $p\colon V^{G}_{\lambda}\to V^{H}_{\mu}$ is the adjoint of $j$, so $p\circ j = I_{V^H_\mu}$. 
Assuming Condition \ref{cond:multfree} the spherical functions of type $\mu$ form a basis of $E^{\mu}$ 
using the algebraic version of the Peter-Weyl Theorem
\cite[Satz 5.2]{DMV-book Springer}. 

In the following subsections we recall some of the properties of the matrix-valued spherical functions and the space of invariants $E^{\mu}$.


\subsection{Orthogonality}\label{subsection:orthogonality}

Note that for the restrictions of $F_1,F_2\in E^{\mu}$ to the compact form $U$, 
the map $U\ni u \mapsto F_1(u)^\ast F_2(u)\in \End(V_{\mu}^H)$ is left $K$-invariant.
Here the adjoint is taken with respect to the inner product on $V^{H}_{\mu}$ for which the 
corresponding $K$-representation is unitary. Then the scalar map 
$U\ni u \mapsto \tr\bigl(F_1(u)^\ast F_2(u)\bigr)$ is $K$-biinvariant.

The space $E^{\mu}$ carries the following Hermitian structure;
\begin{equation}\label{eq:innerproductF1F2}
\langle F_{1},F_{2}\rangle_{\mu}=\int_{U}\tr\bigl( F_{1}(u)^{\ast}F_{2}(u)\bigr)\, du,\qquad F_{1},F_{2}\in E^{\mu},
\end{equation}
where  $du$ is the Haar measure on $U$ normalized by $\int_{U}du=1$. 
By Schur's orthogonality relations the spherical functions $\Phi^{\mu}_{\lambda}$ satisfy the orthogonality relations
\begin{equation}\label{eq:innerproductPhilambdaPhilambda}
\langle\Phi^{\mu}_{\lambda},\Phi^{\mu}_{\lambda'}\rangle_{\mu}=\frac{\dim(V^{H}_{\mu})^{2}}{\dim(V^{G}_{\lambda})}\delta_{\lambda,\lambda'},
\qquad\lambda,\lambda'\in P^{+}_{G}(\mu).
\end{equation}

The integral \eqref{eq:innerproductF1F2} can be reduced using that 
the symmetric pair $(U,K)$ admits a $KAK$-decomposition, see \cite[Ch.X, \S1, no.5]{HelgasonDGSSold}, which is the reference for this subsection.  
Let $\theta\colon G\to G$ be the involution such that $H$ is the connected component of the group of fixed points, $H=(G^{\theta})_{e}$. 
We assume $\theta$ is the complexification of an involution that we denote by the same symbol, 
$\theta\colon U\to U$, for which $K=(U^{\theta})_{e}$. 
Let $\lag$, $\lah$ denote the complex Lie algebras of the groups $G$, $H$, and 
let $\lau$, $\lak$ denote the real Lie algebras of the groups $U$, $K$. 
Let $\lau=\lak\oplus\lap_{c}$ denote the Cartan decomposition of $\lau$ into the $\pm$-eigenspaces of $\theta$. 
Let $\laa_{c}\subset\lap_{c}$ denote a maximal abelian subspace and let $A_{c}\subset K$ denote the connected torus with 
$\Lie(A_{c})=\laa_{c}$. Denote $M_{c}=Z_{K}(\laa_{c})$, 
$\lam_{c}=\Lie(M_{c})$ and let $\lat_{M_{c}}\subset\lam_{c}$ be a maximal torus. 
The complexifications of $M_{c}$, $A_{c}$, $\lam_{c}$, $\laa_{c}$, $\lat_{M_{c}}$ are denoted by $M$, $A$, $\lam$, $\laa$, $\lat_{M}$.

Let $\lag_{0}=\lak\oplus i\lap_{c}$ be the non-compact Cartan dual of $\lau$. 
The tori $\lat_{0}=\lat_{M_{c}}\oplus\laa_{0}\subset\lag_{0}$, with $\laa_{0}=i\laa_{c}$, and 
$\lat=\lat_{M}\oplus\laa\subset\lag$ are maximal. We denote the corresponding root systems by
$$
\Delta=\Delta(\lag,\lat)\quad\mbox{and } \Sigma=\Sigma(\lag_{0},\laa_{0}).
$$
The Weyl groups are denoted by $W(\Delta)$ and $W(\Sigma)$. We fix compatible orderings on the duals of 
$i\lat_{c}$, where $\lat_{c}=\lat_{M_{c}}+\laa_{c}$, 
and $\laa_{0}$ to obtain subsets of positive roots $\Delta^{+}\subset\Delta$ and $\Sigma^{+}\subset\Sigma$. 
Furthermore, denote $P_{+}=\{\alpha\in\Delta^{+} \mid \alpha\ne\alpha\circ\theta\}$ and $P_{-}=\{\alpha\in\Delta^{+} \mid \alpha=\alpha\circ\theta\}$.
The compact group $U$ admits the decomposition $U=KA_{c}K$. 
Note that the dimension of $A_{c}$ is equal to the rank $n$ of the symmetric space. 
The integral over $U$ can be rewritten as
\begin{eqnarray}\label{eq: integral decomp KAK}
\int_{U} f(u)\, du= c_{1}\int_{K}\int_{K}\int_{A_{c}} f(k_1ak_2) \,|\delta(a)|\, da \, dk_{1}\, dk_{2},
\end{eqnarray}
where $\delta(\exp(H))=\prod_{\alpha\in P_{+}}(e^{\alpha(H)}-e^{-\alpha(H)})$.
Recall that $\alpha$ takes purely imaginary values on $\lat_{c}$, so that $\delta$ is the product of sine functions and a constant.
Here, $da$ and $dk$ are the Haar measures on $A_{c}$ and $K$ normalized by $\int_{A_{c}}da=\int_{K}dk=1$. 
The constant $c_{1}$ is the reciprocal of $\int_{A_{c}}|\delta(a)|\, da$. 

The integral in \eqref{eq:innerproductF1F2} can be reduced to integrals over $A_c$.
We now describe how the integrand restricts to $A_{c}$. 
If $F\in E^{\mu}$, then $F\vert_{A_{c}}$ takes values in $\End_{M_c}(V^{H}_{\mu})$, which can be identified with $\bbC^{N}$, 
using Schur's Lemma, since $\pi^H_{\mu}\vert_{M}$ splits multiplicity free, see e.g.~\cite[Prop.~2.4]{HvP}.
Indeed, an element in $\End_{M}(V^{H}_{\mu})$ is a block-diagonal matrix, a block for each irreducible $M$-representation, 
consisting of a multiple of the identity. 
Let $P^+_{M}= \{ \upsilon \in \hat{M}\mid [\pi^H_{\mu}\vert_{M}\colon \pi^M_\upsilon]=1\}$, then 
the size of the block corresponds to the dimension of $\pi^{M}_{\upsilon}$, $\upsilon\in P^+_{M}$. 
The identification is given by sending the block-diagonal matrix to the vector that contains the corresponding multiple of the identity. 
The Hermitian inner product on $\End_{M}(V^{H}_{\mu})$ given by $(A,B)\mapsto\tr(A^{\ast}B)$ transfers to the inner product on $\bbC^{N}$ that is 
given by $(\zeta,z)\mapsto\overline{\zeta}^{t}T^{\mu}z$, where $T^{\mu}$ is the diagonal matrix whose entries are given by the dimensions of the 
corresponding representation spaces $V^{M}_{\upsilon}$, $\upsilon\in P^+_{M}$. Let us denote this identification by $i\colon \End_{M}(V^{H}_{\mu})\to\bbC^{N}$. 
Using this identification, we define the functions $\Psi^{\mu}_{\lambda}\colon A_{c}\to\bbC^{N}$ by 
$a\mapsto i(\Phi^{\mu}_{\lambda}(a))$. We obtain 
\begin{equation}\label{eq: pairing on A}
c_1\int_{A_c} \Psi^{\mu}_{\lambda}(a)^{\ast} T^\mu \Psi^{\mu}_{\lambda'}(a)\, |\delta(a)|\, da 
= \frac{\dim(V^{H}_{\mu})^{2}}{\dim(V^{G}_{\lambda})}\delta_{\lambda,\lambda'},
\qquad\lambda,\lambda'\in P^{+}_{G}(\mu).
\end{equation}

Let $E^{\mu}_{A_{c}}=\{F|_{A_{c}}\mid F\in E^{\mu}\}$ and let $R(A_{c})$ be the algebra of Laurent polynomials on $A_{c}$.
The Weyl group $W=W(\Sigma)$ acts on $E^{\mu}_{A_{c}}$. Indeed, $W$ acts on $A_{c}$ and thereby on 
functions on $A_{c}$, in particular on $R(A_{c})$. The group $W$ also acts on $\End_{M_{c}}(V^{H}_{\mu})$, since $W=N_{K}(\laa_{c})/M_{c}$. 
We obtain an action of $W$ on $R(A_{c})\otimes\End_{M_{c}}(V^{H}_{\mu})$ which is given by
$$
(w\cdot F)(a)=\pi^K_\mu (n_{w}) F(n_{w}^{-1} an_{w}) \pi^K_\mu (n_{w}^{-1}),
$$
where $n_{w}$ represents $w\in W$.
Observe that $(wF)(a)=F(a)$ by \eqref{eqn: trafo} for $F\in E^{\mu}_{A_c}$, 
hence $E^{\mu}_{A_{c}}\subset \left(R(A_{c})\otimes\End_{M_{c}}(V^{H}_{\mu})\right)^{W}$. This inclusion is strict in general, see Remark \ref{rk: strict inclusion}.


\subsection{Differential operators}\label{subsection: DO general}

Let $U(\lag)^{\lah}$ denote the centralizer of $\lah$ in $U(\lag)$. 
The irreducible $H$-representation $(\pi^{H}_{\mu},V^{H}_{\mu})$ induces an irreducible $\lah$-representation 
$(\pi^{\lah}_{\mu},V^{H}_{\mu})$ and thus a representation $U(\lah)\to\End(V^{H}_{\mu})$. 
The kernel of this map is denoted by $I^{\mu}$. 
The equivalence classes of irreducible $\lag$-re\-pre\-sen\-ta\-tions such that the restriction to $\lah$ contains $\pi^{\lah}_{\mu}$ 
are in a one-to-one correspondence with the equivalence classes of the irreducible representations of the algebra
$$
\bbD(\mu)=U(\lag)^{\lah}/I(\mu),\quad I(\mu)=U(\lag)^{\lah}\cap U(\lag)I^{\mu},
$$
see e.g. \cite[Thm.~9.1.12]{Dixmier}. 
Because of Condition \ref{cond:multfree}, the algebra $\bbD(\mu)$ is commutative. Indeed, all the irreducible finite dimensional representations of $\bbD(\mu)$ are one-dimensional. The commutativity also follows from \cite[Thm.3]{Deitmar90}.

Given a smooth $\End(V)$-valued function $F$ on $G$ and an element $X=X_{1}\cdots X_{p}\in U(\lag)$, with $X_i\in\lag$ for all $i$, 
we define $X(F)\colon G\to\End(V^{H}_{\mu})$ by
$$
X(F)(g)=\left(\frac{\del^p}{\del_{t_{1}}\cdots\del_{t_{p}}}F(g\cdot\exp(t_{1}X_{1})\cdots\exp(t_{p}X_{p}))\right)\Bigl\vert_{t_{1}=\ldots=t_{p}=0},
$$
so that $X$ is a left-invariant differential operator. 
We can extend this action linearly, so that $U(\lag)$ can be viewed as an algebra of $G$-left-invariant differential operators. 
Note that for $F\in E^{\mu}$, the function $X(F)$ may not be in $E^{\mu}$. 
However, if $X\in U(\lag)^{\lah}$ then $X(E^{\mu})\subset E^{\mu}$.

The kernel of the representation $U(\lag)^{\lah}\to\End(E^{\mu})$ contains $I(\mu)$, so we obtain an algebra homomorphism $\bbD(\mu)\to\End(E^{\mu})$.

\begin{lemma}\label{lemma: SF sim eigenfunctions}
Let $F\in E^{\mu}$ be a simultaneous eigenfunction of $\bbD(\mu)$. Then $F=c\Phi^{\mu}_{\lambda}$ for a constant $c$ and 
 a unique $\lambda\in P^{+}_{G}(\mu)$.
\end{lemma}

\begin{proof}
The trace of a spherical function is called a $K$-central spherical function. The spherical functions and their traces are related by
$$F(u)=\int_{K}\tr(F(uk^{-1}))\pi^{K}_{\mu}(k)dk,$$
see e.g.~\cite[3.3.26]{MVPthesis}. The result follows from the similar statement for $K$-central spherical functions, see \cite[Thm.~6.1.2.3]{Warner2} or \cite[Thm.~1.4.5]{GangVara}.
\end{proof}

The system of differential equations
$$
D(F)=\gamma_\mu(D,\lambda)F,\quad\mbox{for all $D\in\bbD(\mu)$}
$$
is called the system of hypergeometric differential  equations with spectral parameter $\lambda\in P^{+}_{G}(\mu)$, 
compare to e.g.~\cite[Def.~4.1.1, Def.~5.2.1]{Heckman}. 
In the rank one case, $n=1$, one can show that the differential equation corresponding to the Casimir operator,
see \eqref{eq:radpartCasimir} below, is a so called matrix-valued hypergeometric differential operator, see e.g.~\cite{HvP,tirao} or \cite[Rmk.~3.10]{vPR2}. 

Let $Z(\lag)$ denote the center of $U(\lag)$. One can show that $Z(\lag)\to\bbD(\mu)$ is not surjective in general. 
In fact, already for the case $\mu=0$ it need not be surjective, see \cite[Prop.~5.32]{HelgasonGGAnew}. However, the algebra $\bbD(\mu)$ is 
finitely generated over $Z(\lag)$, see e.g.~\cite[Thm.~9.5.1]{Dixmier}. It is in general difficult to determine the algebra $\bbD(\mu)$, see for example 
\cite[Conj.~10.2,3]{Lepowsky1973}.

Recall that $E^{\mu}_{A_{c}}\subset (R(A_{c})\otimes\bbC^{N})^{W}$.
The $\mu$-radial part of an element $D\in\bbD(\mu)$ is defined to be the operator 
$\rad_{\mu}(D)\in\End((R(A_{c})\otimes\End_{M_c}(V^{H}_{\mu}))^{W})$ such that for all 
$F\in E^{\mu}$, $D(F)\vert_{A_{c}}=\rad_{\mu}(D)(F\vert_{A_{c}})$. 
It turns out that $\rad_{\mu}(D)$ is again a differential operator, see \cite[\S3]{CM1982} or \cite[Ch.~9]{Warner2}. 

The Casimir operator $\Omega$ corresponds to an element of the center of $U(\lag)$, so it gives rise to 
a left-invariant operator for which all the matrix-valued spherical functions are eigenfunctions. 
In order to describe the Casimir operator, 
let $(\xi_{1},\ldots,\xi_{n})$ be an orthonormal basis of $\laa_{c}$ with respect to the Killing form. 
The $\mu$-radial part of the Casimir operator $\Omega$ is given by
\begin{equation}\label{eq:radpartCasimir}
\rad_\mu(\Omega) = \Omega^{\mu}=\sum_{i=1}^{n}\del_{\xi_{i}}^{2}+\pi^{H}_{\mu}(\Omega_{M})+
\sum_{\alpha\in P^{+}}(\alpha,\alpha)\frac{1+e^{-2\alpha}}{1-e^{-2\alpha}}\del_{\alpha^{\vee}}+F^{\mu},
\end{equation}
where $F^{\mu}$ is an $\End(\End_{M_c}(V^{H}_{\mu}))$-valued function that can be calculated explicitly and $\Omega_{M}$ is the quadratic Casmir operator of $M$, see e.g.~\cite[Prop.~9.1.2.11]{Warner2}, 
\cite[p.~881]{CM1982}, \cite[Not.~5.1.3]{Heckman}. 
Note that $\Omega^{\mu} \colon  E^\mu_{A_c}\to E^\mu_{A_c}$. 
Moreover, the spherical functions restricted to $A_c$ are joint eigenfunctions of the Casimir operator, 
\begin{equation*}
\Omega^{\mu} \Phi^\mu_\lambda\vert_{A_c} = \gamma_\mu(\Omega,\lambda) \Phi^\mu_\lambda\vert_{A_c}.
\end{equation*}
We view $\Omega^{\mu}$ acting on 
$\End_{M_c}(V^{H}_{\mu})$-valued Laurent polynomials on $A_c$.
The eigenvalues for the Casimir operator are independent of $\mu$ and $\gamma_\mu(\Omega,\lambda) = |\lambda + \rho|^2 - |\rho|^2$, where the length is with respect to the Killing form and $\rho =\frac12 \sum_{\alpha >0}\alpha$.

\begin{remark}\label{rk: strict inclusion}
Note the embedding $E^{\mu}_{A_{c}}\to\left(R(A_{c})\otimes\End_{M_{c}}(V^{H}_{\mu})\right)^{W}$ is not surjective in general. 
Indeed, in general $F^{\mu}$ is non-constant, so that the constant functions in the space
$\left(R(A_{c})\otimes\End_{M_{c}}(V^{H}_{\mu})\right)^{W}$ cannot be eigenfunctions for the $\mu$-radial 
part $\Omega^{\mu}$ of the Casimir operator.
\end{remark}


\section{Matrix-valued orthogonal polynomials}\label{sec:MVOPs}

We want to associate matrix-valued orthogonal polynomials to the matrix-valued spherical functions by writing 
a general spherical function as
an $E^{0}$-linear combination of a finite number of minimal spherical functions.
For this we have to impose additional conditions to Condition \ref{cond:multfree}. 

Let $P^+_G(0)$ be defined as in \eqref{eq: mu-well} for the trivial representation $\mu=0$, so 
for $\lambda\in P^+_G(0)$ the irreducible holomorphic representation $\pi^G_\lambda$ of 
$G$ contains the trivial $H$-representation exactly once upon restriction to $H$, i.e.~
$[\pi^G_\lambda\vert_H \colon \pi_0^H] =1$. 
We let $\lambda_1, \ldots, \lambda_n$ be the generators for $P^+_G(0)$, where 
$n$ is the rank of the compact symmetric space $U/K$, so 
$P^+_G(0) = \bigoplus_{i=1}^n \bbN_0 \lambda_i$. The corresponding spherical 
functions are $\phi_i = \Phi^0_{\lambda_i}\colon G\to \bbC$, so that 
$\phi_i$ are $H$-biinvariant regular functions on $G$. We then 
have $E^{0}= \bbC[\phi_1,\cdots,\phi_n]$, which we abbreviate as 
$E^{0}= \bbC[\phi]$. As in Subsections \ref{subsection:orthogonality} and \ref{subsection: DO general} it suffices to 
consider $\phi_j$ as a Laurent polynomial on the compact torus $A_c$ and 
then the $\phi_j$ are invariant under the Weyl group $W=W(\Sigma)$. 

Since $P^+_G(0) = \bigoplus_{i=1}^n \bbN_0 \lambda_i$, we can write $\lambda \in P^+_G(0)$ 
uniquely as $\lambda = \sum_{i=1}^n d_i \lambda_i$, $d_i\in \bbN_0$. Define the total degree of $\lambda$ as $|\lambda| = \sum_{i=1}^n d_i$. 

If $\lambda\in P^{+}_{G}(\mu)$ and $\lambda_{\sph}\in P^{+}_{G}(0)$, then $\lambda+\lambda_{\sph}\in P^{+}_{G}(\mu)$.
Indeed, the Borel-Weil Theorem realizes the irreducible $G$-representations in the space of sections of equivariant line bundles over $G/B$, 
\cite[Thm.~4.12.5]{DuistermaatKolk-LieGroups}.
The Cartan projection map $V^{G}_{\lambda_{1}}\otimes V^{G}_{\lambda_{2}}\to V^{G}_{\lambda_{1}+\lambda_{2}}$ is $G$-equivariant and is given by point-wise 
multiplication of algebraic functions, hence is non-trivial.

We impose the following additional structure on the set $P^{+}_{G}(\mu)$. To state it we have to fix a Borel subgroup of $M$ which we choose inside the Borel subgroup of $G$ that we have chosen to fix a notion of positivity. This can always be arranged if we start with a Borel subgroup of $M$ and then extend it to a Borel subgroup of $G$.

\begin{condition}\label{cond:structurePbottom}
Assume that there exists a set of weights $B(\mu)$ for $G$ so that 
for each $\lambda\in P^+_G(\mu)$ there exist 
unique elements $\nu\in B(\mu)$ and $\ls\in P^+_G(0)$ so that $\lambda = \nu+\ls$.
Moreover, we assume that the restriction to $\lat_M$ induces an 
isomorphism $B(\mu) \overset{\cong}{\longrightarrow} P^+_{M} = \{ \upsilon \in P^+_M\mid [\pi^H_\mu\vert_M \colon \pi^M_\upsilon]=1\}$. 
\end{condition}

Note that the isomorphism implies $\# B(\mu)=N$ with $N=\dim_\bbC \End_{M}(V^H_\mu)$ as in Subsection \ref{subsection:orthogonality}. 
In general we have $\# B(\mu)\geq N$ by \cite[Thm.~3.1]{mvp-MVMVOP}.
We put $B(\mu)=\{\nu_1,\cdots, \nu_N\}$ and 
we assume a total order $\nu_1<\nu_2<\cdots<\nu_N$ on $B(\mu)$, which is compatible with the partial order 
on the weights. 

Having observed that $E^{\mu}$ is a module over $E^{0}$ and assuming Condition 
\ref{cond:structurePbottom}, we investigate how the matrix-valued spherical 
functions $\Phi^\mu_{\nu_k}$, $k=1,\cdots, N$, and the $E^0$-module structure of $E^{\mu}$
determine $E^{\mu}$. 
Identify $\bbN^n_0 \to P^+_G(0)$, $d=(d_1,\cdots, d_n) \mapsto \lambda_d = \sum_{i=1}^n d_i \lambda_i$,
so that we can write any element in $\lambda \in P^+_G(\mu)$
as $\lambda = \nu_k + \lambda_d$ for uniquely determined $\nu_k\in B(\mu)$ and $d\in \bbN_0^n$.
Understanding the product $\phi_i \Phi^\mu_\lambda$ requires the understanding of the 
tensor product $V^G_{\lambda_i}\otimes V^G_\lambda$ having $V^G_{\lambda+\lambda_i}$ as 
a constituent. For a finite dimensional holomorphic $G$-representation $\pi^G_\lambda$ of highest weight 
$\lambda$, we let $P(\lambda)$ be the set of weights of $V^{G}_{\lambda}$. We need the set of weights for the 
fundamental spherical representations of highest weights $\lambda_i$ that generate $P^+_G(0)$. Now we can formulate the last condition.

\begin{condition}\label{cond:tensor}
For all weights $\nu\in B(\mu)$ and all generators $\lambda_i$ of $P^+_G(0)$
and all $\eta\in P(\lambda_i)$ such that $\nu +\eta \in P^+_G(\mu)$ we have
by Condition \ref{cond:structurePbottom} a unique $\nu'\in B(\mu)$ such that 
$\nu+\eta = \nu'+\lambda$ with $\lambda\in P^+_G(0)$. 
Then $|\lambda|\leq 1$. 
\end{condition}

Condition \ref{cond:tensor} implies that for $\nu+\ls \in B(\mu) + P^+_G(0)=P^+_G(\mu)$, by Condition 
\ref{cond:structurePbottom}, and for 
arbitrary $\lambda_j$ and $\eta\in P(\lambda_j)$ we have 
$\nu+\ls+\eta = \nu'+\lambda$ with $\lambda\in P^+_G(0)$ and $|\lambda|\leq 1 +|\ls|$. 

Moreover, Condition \ref{cond:tensor} gives control on the matrix-valued spherical functions related to the 
tensor product $V^G_{\lambda_i} \otimes V^G_{\nu+\ls}$, see e.g. \cite[Prop.~(3.2)]{Kuma}. In particular, 
Condition \ref{cond:tensor} implies that there exist constants $c^{p,i}_{j,k}$ so that 
\begin{equation}\label{eq:multMVSFbyspherf}
\phi_i \Phi^\mu_{\nu_p+\ls} = \sum_{k=1}^N \sum_{j=1}^n c^{p,i}_{j,k} \Phi^\mu_{\nu_k+\ls+\lambda_j} 
+ \text{l.o.t.}, 
\qquad c^{p,i}_{i,p}\not= 0,
\end{equation}
where $c^{p,i}_{i,p}\not= 0$ follows from the Cartan projection $V^G_{\lambda_i} \otimes V^G_{\nu_p+\ls} \to V^G_{\lambda_i+\nu_p+\ls}$,
see \cite{Kuma}. 
Here the lower order terms correspond to matrix-valued spherical functions $\Phi^\mu_{\nu_k+\lambda'}$ for some $1\leq k\leq N$ and 
$\lambda'\in P^+_G(0)$ with $|\lambda'|\leq |\ls|$. 

\begin{lemma}\label{lem:polsfromMVSF}
Let $\lambda = \nu_j + \lambda_d\in P^+_G(\mu)$ with $\lambda_d=\sum_{i=1}^n d_i\lambda_i$, then there exist uniquely determined polynomials 
$q^{\mu}_{\nu_{i},\nu_{j};d}$ in $n$-variables of total degree $|d|=\sum_{i=1}^n d_i$ so that 
\[
\Phi_{\nu_{j}+\lambda_d}^{\mu}=\sum_{i=1}^N q^{\mu}_{\nu_{i},\nu_{j};d}(\phi_{1},\cdots,\phi_{n})\, \Phi_{\nu_{i}}^{\mu} \in E^{\mu}.
\]
\end{lemma}

\begin{proof} We invert \eqref{eq:multMVSFbyspherf}, and this gives, with $c^{p,i}_{i,p}\not=0$, 
\[
c^{p,i}_{i,p}\Phi^\mu_{\nu_p+\lambda_d+\lambda_i} = \phi_i \Phi^\mu_{\nu_p+\lambda_d}  - 
\underset{\scriptstyle{(k,j)\not= (p,i)}}{\displaystyle{\sum_{k=1}^N \sum_{j=1}^n}} c^{p,i}_{j,k} \Phi^\mu_{\nu_k+\lambda_d+\lambda_j} 
+ \text{l.o.t.}, 
\]
since the lower order terms are of lesser degree, we can deal with this terms by induction on the total degree $|d|$.
The non-zero terms on the right hand side arise from the occurrence of $V^G_{\nu_k+\ls+\lambda_j}$ in the 
tensor product $V^G_{\lambda_i} \otimes V^G_{\nu_p+\ls}$, which are less in the dominance order than 
$\nu_p+\ls+\lambda_i$. Hence, by induction on the dominance order combined with the induction on the degree, the result follows.
\end{proof}

We define the ordered tuple of spherical functions
\begin{equation}\label{eq:defPhidmu}
\Phi_{d}^{\mu} = (\Phi^{\mu}_{\nu_1+\lambda_{d}},\cdots,\Phi^{\mu}_{\nu_N+\lambda_{d}}), \qquad d\in \bbN_0^n
\end{equation}
which we view as a $(\bbC^{N})^\ast \otimes \End(V^{H}_{\mu}) \cong \Hom(\bbC^{N},\End(V^{H}_{\mu}))$-valued function on $G$,
viewing $(\bbC^{N})^\ast$ as row vectors. Hence we have a natural $\End(\bbC^N)$ action from the right. 
Moreover, the recurrence \eqref{eq:multMVSFbyspherf} gives 
that there exist elements $A_{d',i}^{d}\in\End(\bbC^N)$, $|d'|=|d|+1$, and $B_{d',i}^{d}\in\End(\bbC^N)$, $|d'|\le|d|$, for which
\begin{equation}\label{eqn: recurrence}
\phi_{i}\Phi^{\mu}_{d}=\sum_{|d'|=|d|+1}\Phi^{\mu}_{d'}A_{d',i}^{d}+\sum_{|d'|\le|d|}\Phi^{\mu}_{d'}B_{d',i}^{d}, 
\qquad (A_{d+\delta_j,i}^d)_{k,p} = c^{i,k}_{p,j}, 
\end{equation}
where $\delta_j=(0,\cdots,0, 1,0\cdots,0)\in \bbN_0^n$ with the $1$ at the $j$-th place. 

\begin{lemma}\label{lemma: module full spherical functions}
Let $m\in\bbN_{0}$ and denote $\phi^d = \phi_1^{d_1}\cdots \phi_r^{d_r}\in E^0$. 
The right $\End(\bbC^N)$-modules spanned by the functions 
$\{\Phi_{d}\mid |d|\le m\}$ and $\{\phi^{d}\Phi_{0}\mid |d|\le m\}$ are isomorphic as $\End(\bbC^N)$-modules.
\end{lemma}

\begin{proof}
It is clear that the space spanned by 
$\{\phi^{d}\Phi_{0} \mid |d|\le m\}$ is contained in the space spanned by $\{\Phi_{d}\mid |d|\le m\}$
by \eqref{eq:multMVSFbyspherf}.  
To show equality it is sufficient to prove that the vector spaces spanned by 
$\{\Phi_{\nu_j+ \lambda_{d}}\mid |d|=m,j=1,\cdots,N\}$ and 
$\{\phi^{d}\Phi_{\lambda_{\nu_{j}}}\mid |d|=m,j=1,\cdots,N\}$ have the same dimension. 
The former is of dimension $N\cdot\binom{r+m-1}{m}$ by the algebraic version of the Peter-Weyl Theorem \cite[Satz 5.2]{DMV-book Springer}. 
The latter space is of the same dimension, since the columns of $\Phi_{0}^{\mu}$ are linearly independent, see 
e.g.~\cite[Lemma 6.1]{mvp-MVMVOP}.
\end{proof}

With the notation of Lemma \ref{lem:polsfromMVSF}, we define the matrix-valued polynomials in $n$ variables of degree $d\in\bbN_0^n$ by
\begin{equation}\label{eq:defQdpols}
Q^\mu_d(\phi) = \bigl( q^{\mu}_{\nu_{i},\nu_{j};d}(\phi) \bigr)_{i,j=1}^N, \qquad \phi=(\phi_1,\cdots, \phi_n). 
\end{equation}
Lemma \ref{lem:polsfromMVSF} can be rephrased in the notation \eqref{eq:defPhidmu} as 
\begin{equation}\label{eq:PhidmuQdmu}
\Phi^\mu_d = \Phi^\mu_0 Q^\mu_d(\phi), \qquad 0, d\in \bbN_0^n.
\end{equation}

For later reference we record the following result, where $\End(\bbC^N)[\phi]^{m}$ are the $\End(\bbC^N)$-valued polynomials in $\phi=(\phi_1,\cdots,\phi_n)$ 
of total degree at most $m$. 

\begin{proposition}\label{prop: The Q span}
For any $m\in \bbN_0$, the polynomials $( Q^{\mu}_{d}\mid |d|\le m)$ form a basis for $\End(\bbC^N)[\phi]^{m}$.
\end{proposition}

\begin{proof}
This follows from Lemma \ref{lemma: module full spherical functions} and the fact that the columns of $\Phi^{\mu}_{0}$ are linearly 
independent, since $\Phi_0$ is invertible on a dense subset of $A_c$, see \cite[Lemma 6.1]{mvp-MVMVOP}.
\end{proof}

Because of \eqref{eq:PhidmuQdmu}, we see that the polynomials $Q^\mu_d$ satisfy the same recurrence as the $\Phi^\mu_d$ in 
\eqref{eqn: recurrence}. By Proposition \ref{prop: The Q span} we have two bases for  $\End(\bbC^N)[\phi]^{1}$, 
namely the standard basis $(I, \phi_1 I,\cdots, \phi_n I)$ and $(I,Q^\mu_{\delta_1}, \cdots, Q^\mu_{\delta_n})$.

\begin{corollary}\label{cor: invertible}
The matrix $(A^{0}_{\delta_{i},j})_{1\le i,j\le n}\in \End(\bbC^N)^{n\times n}$ is invertible.
\end{corollary}

\begin{proof} According to \eqref{eqn: recurrence} for the polynomials $Q^\mu_d$ with $d=0$, we see that the 
transition between the two bases is given by the invertible matrix
$$
\begin{pmatrix}
I& B^{0}_{0,1} & \cdots & B^{0}_{0,n}\\
0& A^{0}_{\delta_{1},1} & \cdots & A^{0}_{\delta_{1},n}\\
\vdots & \vdots & \ddots & \vdots\\
0& A^{0}_{\delta_{n},1} & \cdots & A^{0}_{\delta_{n},n}
\end{pmatrix}\in\End(\bbC^N)^{(n+1)\times (n+1)}
$$
and hence the lower right hand part is invertible. 
\end{proof}

Having polynomials associated to the matrix-valued spherical functions, we can transfer the properties of 
the matrix-valued spherical functions of Section \ref{sec:MVSFs} to the matrix-valued polynomials $Q^\mu_d$, $d\in \bbN_0^n$. 


\subsection{Orthogonality}\label{subsection:orthogonalityMVOP}

Using the orthogonality relations \eqref{eq:innerproductPhilambdaPhilambda}, \eqref{eq:innerproductF1F2} we have 
the following relations for the polynomials,
\begin{equation*}
\sum_{i,j=1}^N \int_U \bigl( Q^\mu_d(\phi(u))\bigr)^\ast_{p,i} \, \tr\bigl( (\Phi^\mu_{\nu_i}(u))^\ast \Phi^\mu_{\nu_j}(u)\bigr) 
Q^\mu_{d'}(\phi(u))_{j,q} \, du = \delta_{d,d'} \delta_{p,q}
\frac{\dim(V^{H}_{\mu})^{2}}{\dim(V^{G}_{\nu_p+\lambda_d})},
\end{equation*}
where we use $\phi(u)$ to denote $(\phi_1(u), \cdots, \phi_n(u))$. 
Reducing to the integral over $A_c$, since each term in the integrand is $K$-biinvariant,  we find 
\begin{equation*}
c_1 \sum_{i,j=1}^N \int_{A_c} \bigl( Q^\mu_d(\phi(a))\bigr)^\ast_{p,i} \, \tr\bigl( (\Phi^\mu_{\nu_i}(a))^\ast \Phi^\mu_{\nu_j}(a)\bigr) 
Q^\mu_{d'}(\phi(a))_{j,q} \, |\delta(a)|da = \delta_{d,d'} \delta_{p,q}
\frac{\dim(V^{H}_{\mu})^{2}}{\dim(V^{G}_{\nu_p+\lambda_d})}.
\end{equation*}
Recall that we have $\Phi^\mu_\lambda\colon A_c \to \End_{M_c}(V^H_\mu)$, and the identification 
$i\colon \End_{M_c}(V^H_\mu) \to \bbC^N$ and $\Psi^\mu_\lambda = i \circ \Phi^\mu_\lambda \colon A_c \to \bbC^N$
in Section \ref{subsection:orthogonality}. Now define for $d\in \bbN_0^n$
\begin{equation*}
\Psi^\mu_d \colon A_c\to \End(\bbC^N), \quad a\mapsto (\Psi^\mu_{\lambda_d+\nu_1}(a), \cdots, \Psi^\mu_{\lambda_d+\nu_N}(a))
\end{equation*}
then $\Psi^\mu_d(a) = \Psi^\mu_0(a)Q^\mu_d(\phi(a))$, where $0\in \bbN^n_0$ is a multi-index,  as a matrix product.  
With the notation of \eqref{eq: pairing on A} we get the matrix-valued orthogonality 
relations for the matrix-valued polynomials $Q^\mu_d$ of degree $d\in \bbN_0^n$; 
\begin{gather}
c_1 \int_{A_c} Q^\mu_d(\phi(a))\bigr)^\ast \, (\Psi^\mu_0(a))^\ast T^\mu \Psi^\mu_{0}(a) 
Q^\mu_{d'}(\phi(a)) \, |\delta(a)|da = \delta_{d,d'} H_d, \label{eq:orthoCNonA} \\ 
(H_d)_{p,q} =\delta_{p,q} \frac{\dim(V^{H}_{\mu})^{2}}{\dim(V^{G}_{\nu_p+\lambda_d})}.  \nonumber
\end{gather}
All the matrices have size $N\times N$, and the integral is taken entry-wise.

The integrand of \eqref{eq:orthoCNonA} is Weyl group invariant, so we can view it as the pull-back of a function on the image of 
$\phi\colon A_{c}\to\bbC^{n}$ defined by $a\mapsto(\phi_{1}(a),\ldots,\phi_{n}(a))$. In fact, its image 
$\phi(A_{c})$ is contained in a real form 
$\bbR^{n}\subset\bbC^{n}$.
To perform the change of variables, we invoke the following result from Vretare \cite[L.~3.3]{Vretare1976} which also implies that $\phi(A_{c})\subset\bbR^{n}$ is compact with non-empty interior.

\begin{lemma}\label{lemma: Jacobian}
The Jacobian of the map $\phi\colon A_{c}\to\bbR^{n}$ is given by
$$
j(\exp(H))=c_{2}\cdot\prod_{\alpha\in\Sigma^{+}\backslash\frac{1}{2}\Sigma^{+}}(e^{\alpha(H)}-e^{-\alpha(H)}),
$$
i.e.~the product is taken over the positive restricted roots $\alpha$ with $2\alpha\not\in\Sigma^{+}$, for some $c_{2}\in\bbC^{\times}$. 
\end{lemma}

As we have noted above, we can write $\Psi^{\mu}_{0}(a)^{\ast}\, T^{\mu}\,\Psi^{\mu}_{0}(a)
=W_{\pol}^{\mu}(\phi(a))$, where $W_{\pol}^{\mu}\in\End(\bbC^N)[x]$. 
Lemma \ref{lemma: Jacobian} implies that the scalar weight $|c_{1}^{-1}\delta(a)/j(a)|$ is 
$W(\Sigma)$-invariant, hence it is equal to $w(\phi(a))$ for some function 
$w\colon \phi(A_{c})\to\bbR$. Define $W^{\mu}(x)=W^{\mu}_{\pol}(x)w(x)$. 
A family of matrix-valued orthogonal polynomials with respect to the weight $W^{\mu}(x)$ is a family of matrix-valued polynomials $Q_{d}\in\End(\bbC^{N})$ of multi-degree $d$ that are pair-wise orthogonal with respect to integration against $W^{\mu}(x)$ and which satisfy the properties of Proposition \ref{prop: The Q span}. Orthogonal means that the matrix norm is an invertible matrix. These considerations prove Theorem \ref{thm: MVMVOPs from rep theory}.

\begin{theorem}\label{thm: MVMVOPs from rep theory}
The $Q_{d}^{\mu}\in\End(\bbC^N)[x_1,\cdots, x_n]$, $d\in\bbN_{0}^{n}$, constitute a family of matrix-valued orthogonal polynomials with respect 
to the matrix weight $W^{\mu}$ on the compact set $\phi(A_{c})\in\bbR^{n}$. 
The $\End(\bbC^N)$-valued squared norm of $Q^{\mu}_{d}$ equals $H_d$ as in \eqref{eq:orthoCNonA}.
\end{theorem}
 
The polynomials $\{Q^{\mu}_{d}\mid d\in\bbN_0^n\}$ satisfy the following recurrence relation, 
\begin{equation*}
x_{j}Q^{\mu}_{d}(x)=\sum_{|d'|=|d|+1}Q^{\mu}_{d'}(x)A^{d}_{d',j}+\sum_{|d'|=|d|}Q^{\mu}_{d'}(x)B^{d}_{d',j}+\sum_{|d'|=|d|-1}Q^{\mu}_{d'}(x)C^{d}_{d',j}
\end{equation*}
for some coefficients $A^{d}_{d',j}$, $B^{d}_{d',j}$, $C^{d}_{d',j}$ contained in $\End(\bbC^N)$, where $x=(x_1,\cdots, x_n)$. 
Note that these coefficients follow from \eqref{eqn: recurrence}. We obtain examples of a matrix-valued generalization of the multi-variable orthogonal polynomials from \cite{DunkX}.


\subsection{Differential operators}\label{subsection:diffoperatorsMVOP}


For a $D\in \bbD(\mu)$ the $\mu$-radial part $\rad_{\mu}(D)\in\End((R(A_{c})\otimes\End_{M_c}(V^{H}_{\mu}))^{W})$ can 
be extended to act on functions on $A_c$ taking values in the space $\Hom(\bbC^N, \End_{M_c}(V^H_\mu))$ by acting term-wise.
So on $\Phi^\mu_d\vert_{A_c}$ the action is given by 
\begin{equation*}
\begin{split}
\rad_\mu(D)(\Phi^\mu_d\vert_{A_c}) &= \bigl( \rad_\mu(D)(\Phi^{\mu}_{\nu_1+\lambda_{d}}\vert_{A_c}),\cdots,
\rad_\mu(D)(\Phi^{\mu}_{\nu_N+\lambda_{d}}\vert_{A_c})\bigr) \\
&= \bigl( \gamma_\mu(D,\nu_1+\lambda_{d}) \Phi^{\mu}_{\nu_1+\lambda_{d}}\vert_{A_c},\cdots,
\gamma_\mu(D,\nu_N+\lambda_{d}) \Phi^{\mu}_{\nu_N+\lambda_{d}}\vert_{A_c}\bigr)
\end{split}
\end{equation*}
for $d\in \bbN_0^n$. Consider the $\mu$-radial part $\Omega^\mu$ and the radial (for $\mu=0$) part $\Omega^0$, then 
for a suitable function $Q\colon A_c \to \End(\bbC^N)$,
\begin{equation}\label{eq:radpartCasimironproduct}
\Omega^{\mu}(\Phi^{\mu}_{0}Q)=(\Omega^{\mu}\Phi^\mu_{0})Q+\Phi^{\mu}_{0}\Omega^{0}(Q)+2\sum_{i=1}^{r}(\del_{\xi_{i}}\Phi^{\mu}_{0})(\del_{\xi_{i}}Q).
\end{equation}
This follows since in the scalar differential operator $\Omega^0$ we have $F^0=0$, and $F^\mu$ commutes with multiplication
from the right by $\End(\bbC^N)$-valued function. In \eqref{eq:radpartCasimironproduct} 
we use $\Omega^{0}(Q)= \bigl( \Omega^0 Q_{i,j}\bigr)_{i,j=1}^N$ entry-wise.  

We now proceed to rewrite \eqref{eq:radpartCasimironproduct} as a differential operator for $Q$.
For this we conjugate $\Omega^\mu$ by $\Phi^{\mu}_0$, which is invertible on a dense subset of $A_c$, see \cite[Lemma 6.1]{mvp-MVMVOP}, 
and for this we need a first order differential equation for $\Phi^{\mu}_0$. 

\begin{lemma}
\label{lem:firstorderDEsPhi0} 
For all $k=1,\cdots,n$, we have as $\Hom(\bbC^N, \End_{M_c}(V^H_\mu))$-valued 
functions on $A_c$ 
\begin{equation*} 
2\sum_{i=1}^{n}(\del_{\xi_{i}}\Phi^{\mu}_{0})(\del_{\xi_{i}}\phi_{k})=\Phi^{\mu}_{0}(L_{k}(\phi)+C_{k}),
\end{equation*}
where $L_{k}$ is a $\End(\bbC^N)$-valued polynomial in $\phi=(\phi_1,\cdots,\phi_n)$ of degree $1$ without constant term and
$C_k\in \End(\bbC^N)$ is a constant. 
\end{lemma}

\begin{remark}
The function $\Phi_{0}$ is possibly not of full rank in the points where the matrix $(\del_{\xi_{i}}\phi_{k})_{i,k}$ is singular. 
In case $n=1$ this is on the end points of the interval $[-1,1]$, in the cases $n=2,3$ this is on the boundaries of the regions in Figure \ref{fig:2d}.
\end{remark}

\begin{proof}
Note that $F^{0}=0$ and that the functions $\phi_{j}$ and $\Phi^{\mu}_{0}$ are eigenfunctions of $\Omega^{0}$ and $\Omega^{\mu}$ respectively, with eigenvalues $\gamma_{j}=\gamma_0(\Omega,\lambda_j)\in \bbC$ and $\Gamma_{0}=\text{diag}\bigl(\gamma_\mu(\Omega,\nu_1),\cdots,\gamma_\mu(\Omega,\nu_N)\bigr)\in \End(\bbC^N)$ respectively. Similarly we define $\Gamma_{\delta_{i}}=\text{diag}\bigl(\gamma_\mu(\Omega,\nu_1+\lambda_{i}),\cdots,\gamma_\mu(\Omega,\nu_N+\lambda_{i})\bigr)\in \End(\bbC^N)$, the diagonal eigenvalue of the Casimir operator for $\Phi^{\mu}_{\delta_{i}}$.
If we plug in $Q=\phi_{k}I$ in \eqref{eq:radpartCasimironproduct} then we obtain
$$
\Omega^{\mu}(\Phi^{\mu}_{0}\phi_{k})=\Phi^{\mu}_{0}\Gamma_{0}\phi_{k}+
\Phi^{\mu}_{0}\gamma_{k}\phi_{k}+2\sum_{i=1}^{n}(\del_{\xi_{i}}\Phi^{\mu}_{0})(\del_{\xi_{i}}\phi_{k}).
$$
On the other hand, if we apply $\Omega^{\mu}$ to \eqref{eqn: recurrence} for $d=0$, we can evaluate the left hand side.
This gives 
\begin{eqnarray*}
\sum_{i=1}^{n}\Phi^{\mu}_{\delta_{i}}\Gamma_{\delta_{i}}A_{k,\delta_{i}}^{0}+\Phi^{\mu}_{0}\Gamma_{0}B_{k,0}^{0}=
\Phi^{\mu}_{0}\Gamma_{0}\phi_{k}+\Phi^{\mu}_{0}\gamma_{k}\phi_{k}+2\sum_{i=1}^{n}(\del_{\xi_{i}}\Phi^{\mu}_{0})(\del_{\xi_{i}}\phi_{k}).
\end{eqnarray*}
Now use $\Phi^{\mu}_{\delta_{i}}=\Phi^{\mu}_{0}Q^\mu_{\delta_i}(\phi)$, see \eqref{eq:PhidmuQdmu}, and collect the terms.  
\end{proof}

To conjugate the differential operators it is more convenient to work with the functions $\Psi_{0}^{\mu}$, because their values are square matrices. The chain rule implies 
$2\sum_{i=1}^{n}(\del_{\xi_{i}}\Psi^{\mu}_{0})(\del_{\xi_{i}}Q(\phi))=2\sum_{k=1}^{n}\sum_{i=1}^{n}(\del_{\xi_{i}}\Psi^{\mu}_{0})(\del_{\xi_{i}}\phi_{k})\del_{k}Q(\phi)$ 
and together with Lemma \ref{lem:firstorderDEsPhi0} we obtain
$$
(m_{(\Psi^{\mu}_{0})^{-1}}\circ\Omega^{\mu}\circ m_{\Psi^{\mu}_{0}})(Q)(\phi)=\Omega^{0}Q(\phi)+2\sum_{k=1}^{n}(L_{k}(\phi)+C_{k})(\del_{k}Q)(\phi)+\Gamma_{0}Q(\phi),
$$
where $m_{\Psi_{0}^{\mu}}$ denotes multiplication by $\Psi^{\mu}_{0}$ on the right. 
Note that $(\Psi_{0}^{\mu})^{-1}$ exists on a dense subset of $A_c$, see \cite[Lemma 6.1]{mvp-MVMVOP}.
The final manipulation is a change of variables $x=\phi(a)$ for which we need the following identities,
$$
\del_{\xi_{i}}^{2}(Q(\phi))=\sum_{k=1}^{n}\left(\sum_{\ell=1}^{n}(\del_{\ell}\del_{k}Q)(\phi)(\del_{\xi_{i}}\phi_{\ell})(\del_{\xi_{i}}\phi_{k})
+(\del_{k}Q)(\phi)(\del_{\xi_{i}}^{2}\phi_{k})\right)
$$
and 
$$
\sum_{\alpha\in P^{+}}(\alpha,\alpha)\frac{1+e^{-2\alpha}}{1-e^{-2\alpha}}\del_{\alpha^{\vee}}(Q(\phi))=\\
\sum_{i=1}^{n}\left(\sum_{\alpha\in P^{+}}(\alpha,\alpha)
\frac{1+e^{-2\alpha}}{1-e^{-2\alpha}}\del_{\alpha^{\vee}}\phi_{i}\right)(\del_{i}Q)(\phi).
$$ 
This yields
$$
\Omega^{0}(Q(\phi))=\sum_{1\le k,\ell\le r}\left(\sum_{i=1}^{n}(\del_{\xi_{i}}\phi_{\ell})(\del_{\xi_{i}}\phi_{k})\right)(\del_{k}\del_{\ell}Q)(\phi)
+\sum_{k=1}^{n}\gamma_{k}(\del_{k}Q)(\phi).
$$
Finally we obtain
\begin{gather}
(m_{\Phi_{0}^{-1}}\circ\Omega^{\mu}\circ m_{\Phi_{0}})(Q)(\phi)= \label{eq:CasimironQx} \\
\sum_{1\le k,\ell\le n}\left(\sum_{i=1}^{n}(\del_{\xi_{i}}\phi_{\ell})(\del_{\xi_{i}}\phi_{k})\right)(\del_{k}\del_{\ell}Q)(\phi)
+2\sum_{k=1}^{n}(L_{k}(\phi)+C_{k}+\gamma_{k})(\del_{k}Q)(\phi)+\Gamma_{0}Q(\phi). \nonumber
\end{gather}
So \eqref{eq:CasimironQx} gives a second order differential operator $D_\Omega\in \End(\bbC^N)[x,\partial_x]$ having the 
polynomials $Q^\mu_d(x)$, $x=(x_1,\cdots,x_n)$, as eigenfunctions.

For the $\mu$-radial part of the Casimir operator we have an explicit expression. In general we don't have such expressions available. However, in principle we can perform the above construction for any element in $\bbD(\mu)$.

Letting the $\mu$-radial part of an element $D\in\bbD(\mu)$ 
act on $\Psi^\mu_0 Q(\phi)$ for a function $Q$ in $n$ variables, and conjugating by $\Psi^\mu_0$ and changing to coordinates 
$x$, we obtain a differential operator $\End(\bbC^N)[x,\partial_x]$ having the polynomials $Q^\mu_d$ (as function of $x$) as
eigenfunctions. We denote the image of this map $\calD^{\mu} \colon \bbD(\mu) \to \End(\bbC^N)[x,\partial_x]$ by $\calD(\mu)$, which is a commutative algebra of matrix-valued differential operators having the polynomials $Q^\mu_d$ as 
simultaneous eigenfunctions. 

In fact, by Lemma \ref{lemma: SF sim eigenfunctions} the polynomials $Q^{\mu}_{d}$ are determined as simultaneous eigenfunctions of the elements in $\calD(\mu)$. The image of the Casimir operator in $\calD(\mu)$ is also symmetric. Indeed, its eigenvalues are real diagonal matrices and the matrix norms of the polynomials $Q^{\mu}_{d}$ are also diagonal.

To describe another important property of the elements in $\calD(\mu)$ we need the following notation. 
A multi-index $\alpha\in\bbN_{0}^{n}$ has total degree $|\alpha|=\alpha_{1}+\cdots+\alpha_{n}$. 
Given such a multi-degree $\alpha$, we write $\del_{x}^{\alpha}=\del_{x_{1}}^{\alpha_{1}}\cdots\del_{x_{n}}^{\alpha_{n}}$. 

\begin{proposition}\label{prop: do total degrees}
The differential operators in $\calD(\mu)$ are of the form 
$\sum_{k=0}^{t}\sum_{\alpha:|\alpha|=k}P_{\alpha}(x)\del_{x}^{\alpha}$, where $\alpha\in\bbN_{0}^{n}$ and $P_{\alpha}\in\End(\bbC^{N})[x]$ 
is of total degree at most $|\alpha|$.
\end{proposition}

\begin{proof}
A differential operator from $\calD(\mu)$ preserves  polynomials, since the $Q^\mu_d$ are eigenfunctions, 
see Proposition \ref{prop: The Q span}. Hence the coefficients are polynomials. 
Since the $Q^\mu_d$ are eigenfunctions it also preserves the total degree of
these polynomials. This gives the statement on the degree of the polynomials.
\end{proof}

Applying Proposition \ref{prop: do total degrees} to the image $D_\Omega\in \calD(\mu)$ 
of the Casimir operator of \eqref{eq:CasimironQx} gives the following corollary. 

\begin{corollary}
The expression $\sum_{i=1}^{n}(\del_{\xi_{i}}\phi_{\ell})(\del_{\xi_{i}}\phi_{k})$ in \eqref{eq:CasimironQx} is a polynomial of total degree at most two.
\end{corollary}



\part{The case $(U,K)=(\SU(n+1)\times\SU(n+1),\diag\,\SU(n+1))$}\label{part:specialcase}

In this part we adopt the following notation. The pair $(U,K)$ is equal to $(\SU(n+1)\times\SU(n+1),\diag\,\SU(n+1))$ and the pair $(G,H)$, its complexification, 
is equal to $(\SL(n+1,\bbC)\times \SL(n+1,\bbC), \diag\,\SL(n+1))$. Note $\Psi^{\mu}_{0}=\Phi^{\mu}_{0}$, since $M=Z_{K}(A_{c})$ is a maximal torus in $K$.


\section{Structure theory and zonal spherical functions}\label{sec:SLn-structurethy}

Both $(U,K)$ and $(G,H)$ are symmetric pairs, where the involutive 
automorphims are given by the flip $\theta(x,y)=(y,x)$. 
The Lie algebra $\lag$ decomposes according to the $\pm$-eigenspace of the differential of 
$\theta$, $\lag=\lah+\lap$, with $\lap$ isomorphic to $\lah$ as a $\bbC$-vector space.

Let $T\subset\SL(n+1,\bbC)$ be the maximal torus consisting of diagonal elements. The maximal tori of $G$ and $H$ are 
$T_{G}=T\times T$ and $T_{H}=\diag(T)$. 
Let $A=\{(t,t^{-1})\mid t\in T\}$. Then the Lie algebra $\laa$ of $A$ is a maximal abelian subspace of $\lap$, 
whose centralizer in $H$ is $T_{H}$. The root system $\Delta(\lasl_{n+1},\lat)$ is described in \cite[Planche I]{Bourbaki-Lie 4-5-6} 
and we take the same choices here. 
The set of positive roots and simple roots are denoted by $\Delta^{+}(\lasl_{n+1},\lat)$ and $\Pi(\lasl_{n+1},\lat)$ respectively.

The set of roots for $(\lag,\lat_{G})$ is given by $\Delta=\{(\alpha,0) \mid \alpha\in\Delta(\lasl_{n+1},\lat)\}\cup\{(0,\alpha) \mid \alpha\in\Delta(\lasl_{n+1},\lat)\}$. 
We fix the set of positive roots $\Delta^{+}=\{(\alpha,0)\mid \alpha\in\Delta^{+}(\lasl_{n+1},\lat)\}\cup\{(0,-\alpha)\mid\alpha\in\Delta^{+}(\lasl_{n+1},\lat)\}$. 
The corresponding set of simple roots is given by $\Pi=\{(\alpha,0)\mid\alpha\in\Pi(\lasl_{n+1},\lat)\}\cup\{(0,-\alpha)\mid\alpha\in\Pi(\lasl_{n+1},\lat)\}$.

The restricted roots are given by the restrictions of the roots in $\Delta$ to the anti-diagonal $\laa$ in $\lat\oplus\lat$. 
The set of restricted roots is denoted by $\Sigma$. Note that $(\alpha,0)|_{\laa}=(0,-\alpha)|_{\laa}$, which shows that the root multiplicities 
are two, i.e.~the restricted root spaces are two-dimensional. 
The set of positive restricted roots is given by $\Delta^{+}=\{(\alpha,0)|_{\laa}\mid \alpha\in\Delta^{+}(\lasl_{n+1},\lat)\}$. 
The corresponding Weyl group is $W(\Sigma)=S_{n+1}$. 
Moreover, since the flip $\theta$ does not stabilize any root, we have $P^{+}=\Delta^{+}$, where $P^+$ is as in Subsection \ref{subsection:orthogonalityMVOP}.

Upon the identification  $G/H\to\SL(n+1,\bbC)$ induced by the map $(g_1,g_2)\mapsto g_1g_2^{-1}$ the zonal spherical 
functions correspond to a multiple of the characters of the irreducible representations of $\SL(n+1,\bbC)$, 
the multiple being the reciprocal of dimension of the representation, i.e. 
$\phi_{(\lambda,-\lambda)}(x,y) = (\dim(V_\lambda))^{-1} \chi_\lambda(xy^{-1})$ where $\lambda$ is a dominant integral weight for $\SL(n+1,\bbC)$, $V_\lambda$ 
the corresponding finite-dimensional holomorphic representation, and $\chi_\lambda$ its character. 
So in this case $P_G^+(0)=\{ (\lambda,-\lambda) \mid \lambda\in P^{+}_{\SL(n+1,\bbC)}\}$ and 
the fundamental spherical weights of $G$ are given by $\lambda_{i}=(\omega_{i},-\omega_{i})$ where 
$\omega_i$, $i=1,\cdots,n$, are the fundamental weights for $\SL(n+1,\bbC)$, 
which can deduced from the Cartan-Helgason theorem \cite[Thm.~8.49]{KnappLGBI}. Moreover, the trivial representation 
occurs with multiplicity one in the tensor product decomposition, so Condition \ref{cond:multfree} is satisfied. 
This also follows from the fact that $(G,H)$ is a spherical pair, see the first paragraph of Section \ref{sec:invertingbranchingrule}. 

The restriction of the corresponding zonal spherical functions to $A$ are $S_{n+1}$-invariant, 
so they are classical symmetric functions in $n+1$ variables $t=(t_1,\cdots, t_{n+1})$ with the restriction $t_1\cdots t_{n+1}=1$.
We record the explicit expressions of the fundamental zonal spherical functions.

Let $V=\bbC^{n+1}$, equipped with standard orthonormal basis $(e_1,\cdots, e_{n+1})$,  
be the representation space of the standard representation $\pi^{\SL(n+1,\bbC)}_{\omega_{1}}$. 
The representation space of $\pi^{\SL(n+1,\bbC)}_{\omega_{i}}$ is then given by $\bigwedge^{i}V$.

\begin{lemma}\label{lemma: zonal SF}
The zonal spherical function $\phi_i=\phi_{\lambda_i}$ associated to the fundamental spherical weight 
$\lambda_i=(\omega_i,-\omega_i)$ is given by
$$
\phi_{i}(t,t^{-1})=\binom{n+1}{i}^{-1}\sum_{J}t^{2}_{j_{1}}\cdots t^{2}_{j_{i}},
$$
where the sum is taken over all $i$-tuples $1\le j_{1}<\cdots< j_{i}\le n+1$ and for any $1\leq i \leq n$. 
\end{lemma}

Note that the zonal spherical function $\phi_i=\phi_{\lambda_i}$ are invariant under the action of 
the symmetric group $W(\Sigma)$ and under $(t,t^{-1})\mapsto (-t,-t^{-1})$ which 
corresponds to the nontrivial element of $M_c\cap A_c = \{\pm (I,I)\}$. 

\begin{proof}
This follows immediately from $\phi_i(t,t^{-1}) = (\dim(V_{\omega_i}))^{-1} \chi_{\omega_i}(t^2)$ and the 
explicit expressions for the dimension and the character using Weyl's formulas, but we do it more directly. 
The representation space of the spherical representation $\pi^{G}_{(\lambda,-\lambda)}$ is given by 
$V_\lambda \otimes V_\lambda^\ast\cong \End(V^\lambda)$ and then the $G$-representation is 
$(x,y)\cdot A=\pi^{\SL(n+1,\bbC)}_{\lambda}(x) A \pi^{\SL(n+1,\bbC)}_{\lambda}(y^{-1})$. 
Then the identity $I$ is a $H$-fixed vector, and with the 
inner product given by $(A,B)\mapsto\dim(V_\lambda)^{-1}\tr(A^{\ast}B)$, the 
zonal spherical function, as the corresponding matrix entry,  is given by the normalized character.
Now take $\lambda = \lambda_i=(\omega_i,-\omega_i)$, so that 
$V_{\lambda_i}=\End(\bigwedge^{i}V)$, with standard basis elements $e_{j_1}\wedge\cdots \wedge e_{j_i}$ for $1\le j_{1}<\cdots< j_{i}\le n+1$.
\end{proof}

Note that the fundamental spherical functions satisfy $\phi_{i}\circ\theta=\phi_{n+1-i}$ for $i=1,\cdots,n$, which follows from the more general rule $\Phi^{\mu}_{\lambda}(\theta(g))=\Phi^{\mu^{*}}_{\lambda^{*}}(g)^{*}$. 
This implies $\phi_{i}(t,t^{-1})=\overline{\phi_{n+1-i}(t,t^{-1})}$ for $(t,t^{-1})\in A_{c}=A\cap(\SU(n+1)\times \SU(n+1))$. 
Hence the image $\phi(A_{c})$ is contained in the real space 
$\bbR^{n}=\{(z_{1},\ldots,z_{n})\in\bbC^{n}:z_{i}=\overline{z}_{n+1-i}\}$. 

Let $(e_{1},\ldots,e_{n})$ be the standard basis of $\bbC^{n}$. 
Let $F\in\bbR[z_{1},\ldots,z_{n}]$ be a polynomial viewed as a function on $\bbR^{n}$, 
i.e.~where $z_{i}(e_{j})=\delta_{i,j}$. Our aim is to write the integral
$\int_{U}F(\phi(u))\, du$, $U=\SU(n+1)\times \SU(n+1)$, 
with $du$ the Haar measure normalized by $\int_{U}du=1$, 
as an integral of $F$ over $\phi(A_{c})$. We proceed in four steps.

(1) Using the decomposition of the integral for the $U=KA_{c}K$-decomposition, see \eqref{eq: integral decomp KAK}, we obtain
$$
\int_{U}F(\phi(u))\, du=c_{1}\int_{A_{c}}F(\phi(a))|\delta(a)|\, da,
$$
where $\delta(\exp(H),\exp(-H)) = \prod_{\alpha\in\Delta^{+}(\lasl_{n+1},\lat)}(e^{\alpha(H)}-e^{-\alpha(H)})^{2}$. 
In order to calculate $c_{1}$ we have to evaluate a Selberg integral
\begin{equation}\label{eq: selberg integral}
\int_{A_{c}}|\delta(a)|^{s}\ da=\frac{\Gamma(1+(n+1)s)}{\Gamma(1+s)^{n+1}},
\end{equation}
for $s=1$, see e.g.~\cite{ForrW} or \cite[Ex.3.5.8]{Heckman}. Hence $c_1 =((n+1)!)^{-1}$. 

(2) We identify $\laa_c = \{ (H,-H) \mid H = i(h_1,\cdots, h_{n+1})\in i\bbR^{n+1}, \sum_{k=1}^{n+1} h_k =0\}$. 
By abuse of notation we use $\alpha = (\alpha, 0)\vert_{\laa_c}\in \Sigma$ and $\omega_{i}=(\omega_{i},0)|_{\laa_{c}}$. Let $\alpha^\vee\in \laa_c$ be the coroot, i.e.~$\alpha_{i}^{\vee}$ is identified with $i(e_{i}-e_{i+1})$. Then the Haar measure on $A_c$ is the push forward of the 
form $(2\pi)^{-n}d\omega_1\wedge \cdots \wedge d\omega_n$ under the exponential map on 
$\mathfrak{f}= \{ \sum_{k=1}^n s_k\alpha_k^\vee \mid 0\leq s_k < 2\pi\}$ since $\omega_k(\alpha_l^\vee)=\delta_{k,l}$, i.e.
\[
\int_{A_c} f(a) \, da = \frac{1}{(2\pi)^n} \int_{\mathfrak{f}} f\bigl( \exp(H), -\exp(H)\bigr) \, d\omega_1\wedge \cdots \wedge d\omega_n. 
\]
Note that $\mathfrak{f}$ is a fundamental domain for the translations by $2\pi \Lambda_{Q^\vee}$,  where $\Lambda_{Q^\vee}$ is the 
coroot lattice. 

(3) Since $a\mapsto|\delta(a)|$ is Weyl group invariant, the integrand is Weyl group-invariant. Note that a fundamental 
domain for the action of $W$ on $\mathfrak{f}$ mod the action of $2\pi \Lambda_{Q^\vee}$ is given by the 
fundamental alcove $\mathfrak{b}$ in the Stiefel diagram, see \cite[\S3.11]{DuistermaatKolk-LieGroups}. 
Then $\mathfrak{b} = \{ \sum_{k=1}^n b_k \omega_k^\vee\mid b_k\geq 0, k=1,\cdots, n, \sum_{k=1}^n b_k \leq 2\pi \}$,
where $\omega_k^\vee \in \laa_c$ is defined by $\alpha_l(\omega_k^\vee)=\delta_{k,l}$, and we obtain
\begin{multline*}
\frac{1}{(n+1)!} \int_{A_c} F(\phi(a))\, |\delta(a)| \, da = \\
\frac{1}{(2\pi)^n} 
\int_{\mathfrak{b}} F\bigl( \phi(\exp(H), -\exp(H))\bigr)  |\delta(\exp(H), -\exp(H))| \, d\omega_1\wedge \cdots \wedge d\omega_n. 
\end{multline*}

(4) We observe that $\delta(a)=P(\phi(a))$ for some polynomial $P$, since $a\mapsto\delta(a)$ is invariant under the action of 
$W$ and the action of $M_c\cap A_c=\{\pm(I,I)\}$. The Jacobian in Lemma \ref{lemma: Jacobian} is the square root of 
$|\delta(a)|$ in this case. Up to the constant factor we have proved the following result.

\begin{lemma}\label{lem:constantintAccorrect} With the notation from this section we have
\begin{equation}\label{eq:expressionintAcasintegralphi}
\frac{1}{(n+1)!} \int_{A_c} F(\phi(a))\, |\delta(a)| \, da = 
\frac{1}{(2\pi)^n} \left( \prod_{k=1}^n \binom{n+1}{k} \right)
\int_{\phi(\exp(\mathfrak{b}))} F(\phi)  |P(\phi)|^{\frac12} \, d\phi,
\end{equation}
where $d\phi = d\phi_1\wedge \cdots\wedge d\phi_n$. 
\end{lemma}

\begin{proof} 
It remains to show that the constant $\prod_{k=1}^n \binom{n+1}{k}$ in \eqref{eq:expressionintAcasintegralphi} is correct. 
Since the coroots $\alpha^\vee_k$ are dual to the fundamental weights $\omega_l$ it suffices to 
take the partial derivatives of the fundamental spherical functions with respect to the coroots.
Note that $\phi_k(\exp(H), \exp(-H)) =\binom{n+1}{k}^{-1} e^{2\omega_k(H)} + \text{l.o.t}$, 
where the lower order terms are with respect to the partial order. 
So the determinant of $\bigl(\frac{\partial\phi_k}{\partial \alpha^\vee_l}\bigr)_{1\leq k,l\leq n}$ 
is of the form $ 2^n \prod_{k=1}^n \binom{n+1}{k}^{-1} e^{2\rho} + \text{l.o.t}$, with $\rho = \sum_{k=1}^n \omega_k 
= \frac12 \sum_{\alpha>0} \alpha$, as only the diagonal elements in the matrix contribute to 
the coefficient of $e^{2\rho}$. 
By Lemma \ref{lemma: Jacobian} the coefficient of the leading term $e^{2\rho}$ in $j$ in this case is $c_2 (2i)^n$. 
Taking absolute values and comparing 
the constants determines the value of $|c_2|$. 
Observe that $|j(\phi)|=|c_2| |P(\phi)|^{\frac12}$, so that 
$|\delta(\phi(a))|/|j(\phi)| = |c_2|^{-1} |P(\phi)|^{\frac12}$, and \eqref{eq:expressionintAcasintegralphi}
follows. 
\end{proof}

Note that $\phi(\exp(\mathfrak{b}))=\phi(A_c)$. We record the following special case of 
\eqref{eq:expressionintAcasintegralphi} in conjuction with the Selberg integral \eqref{eq: selberg integral},
$$\int_{\phi(A_c)} |P(\phi)|^{s} \, d\phi = \frac{(2\pi)^n}{\prod_{k=1}^n \binom{n+1}{k}}
\frac{1}{(n+1)!} \frac{\Gamma(1+(n+1)(s+\frac{1}{2}))}{\Gamma(\frac{3}{2}+s)^{n+1}},
$$
which leads to an expression for the volume of $\phi(A_{c})$,
$$\text{vol}\bigl(\phi(A_c)\bigr)= \int_{\phi(A_c)} \, d\phi = \frac{(2\sqrt{\pi})^n}{\Gamma(1+\frac{n}{2})\prod_{k=1}^n \binom{n+1}{k}}.$$

For $n=2$ we obtain the area of Steiner's hypocycloid, which is $4\pi/9$. For $n=3$ we obtain the volume of the 3-dimensional analog of Steiner's hypocycloid, which equals $\pi/9$. See Figure \ref{fig:2d}.

Now that we have \eqref{eq:expressionintAcasintegralphi} it remains to study  
the polynomial $P$ and $\phi(A_c)$. First note that $\delta(a)=P(\phi(a))$ and $\phi(A_c) = \phi(\exp(\mathfrak{b}))$, which
shows that $P$ vanishes at the boundary of $\phi(A_c)$ and is non-zero in the interior
since $H \mapsto \delta(\exp(H), \exp(-H))$ vanishes at the boundary of $\mathfrak{b}$ and is non-zero at its interior. 

\begin{lemma}\label{lem:barycenter}
The barycenter $H_0$ of the fundamental alcove $\mathfrak{b}$ is mapped to $0\in \bbC^n$ by $\phi\circ\exp$. In particular, $0$ is contained in the interior of $\phi(A_c)$. 
\end{lemma}

\begin{proof} $H_0=\frac{\pi}{n+1}\sum_{k=1}^n \omega^\vee_k=\frac{\pi i}{n+1}(\frac12n, \frac12(n-2), \cdots, -\frac12 n)$,
so that $t_0=\exp(H_0) = (\exp(\frac{in\pi}{2(n+1)}), \cdots,\exp(-\frac{in\pi}{2(n+1)}))$ and 
$\binom{n+1}{i}\phi_i(t_0,t_0^{-1}) =e_i(t^2_0)$, where $e_i$ is the $i$-th elementary symmetric function, see Lemma 
\ref{lemma: zonal SF}. The generating function for the elementary symmetric function gives, see also \eqref{eq:genfunelemtarysymmf},
\begin{equation*}
\prod_{k=1}^{n+1}(z- e^{\frac{i\pi (n-2k)}{n+1}}) = z^{n+1} -e_1(t_0^2)z^n + e_2(t_0^2)z^{n-1} - \cdots 
+(-1)^{n} e_{n}(t_0^2) + (-1)^{n+1} e_{n+1}(t_0^2)
\end{equation*}
and $e_{n+1}(t_0^2)=1$. Since the polynomial $z^{n+1}+(-1)^{n+1}$ has 
the same zeros $\bigl\{ e^{\frac{i\pi (n-2k)}{n+1}}\mid k=0,\ldots, n\bigr\}$ we see that 
$e_k(t_0^2)=0$ for $k=1,\ldots,n$. 
\end{proof}

It follows that the image $\phi(A_{c})$ is the closure of the connected component of the set $\{v\in\bbR^{n} \mid P(v)\ne0\}$ that contains $0$.

\begin{lemma}\label{lem:determinationP}
Let $p_{k}(t_{1},\ldots,t_{n+1})=t_{1}^{k}+\cdots+t_{n+1}^{k}$ be the symmetric power sum.
Then $\det(p_{i+j-2}(t^2))_{1\leq i,j\leq n+1}=\delta(t,t^{-1})=P(\phi(t,t^{-1}))$ for some polynomial $P\in\bbR[z_{1},\ldots,z_{n}]$.
\end{lemma}

This result can be used to explicitly determine $P$ using the Newton-Girard formulas expressing the symmetric power
sums in the elementary spherical function, see \cite[\S 10.12]{seroul}.

\begin{proof} Observe that 
$\delta(t,t^{-1}) = \prod_{1\leq i<j\leq n+1} (\frac{t_i}{t_j}- \frac{t_j}{t_i})^2$. Taking the 
common denominator out of the product, we have, using that $t_1t_2\cdots t_{n+1}=1$,
$\delta(t,t^{-1}) = \prod_{1\leq i<j\leq n+1} (t_i^2- t_j^2)^2$. By Vandermonde's determinant
this equals $(\det A)^2$ for the 
$(n+1)\times (n+1)$-matrix $A$ with $A_{i,j}= t_j^{2(i-1)}$. Note that 
$(A^tA)_{i,j} = \sum_{k=1}^{n+1} t_k^{2(i+j-2)} = p_{i+j-2}(t^2)$,
so that $\delta(t,t^{-1}) = \det (A^tA)$ gives the result.
\end{proof}

We summarize these results in the following theorem.

\begin{theorem}
\label{thm:scalar_weight}
Let $F\in\bbR[z_{1},\ldots,z_{n}]$. Then
$$
\int_{U}F(\phi(u))du=
\frac{1}{(2\pi)^n} \left( \prod_{k=1}^n \binom{n+1}{k} \right)
\int_{\phi(\exp(\mathfrak{b}))} F(\phi)  w(\phi) \, d\phi,
$$
where $w(z)=|P(z)|^{1/2}$. Moreover, $\phi(A_{c})$ is equal to the closure of the connected component of 
$\{v\in\bbR^{n}\mid P(v)\ne0\}$ that contains 0.
\end{theorem}

\begin{figure}[t]
\centering
\begin{overpic}[width=.35\textwidth]{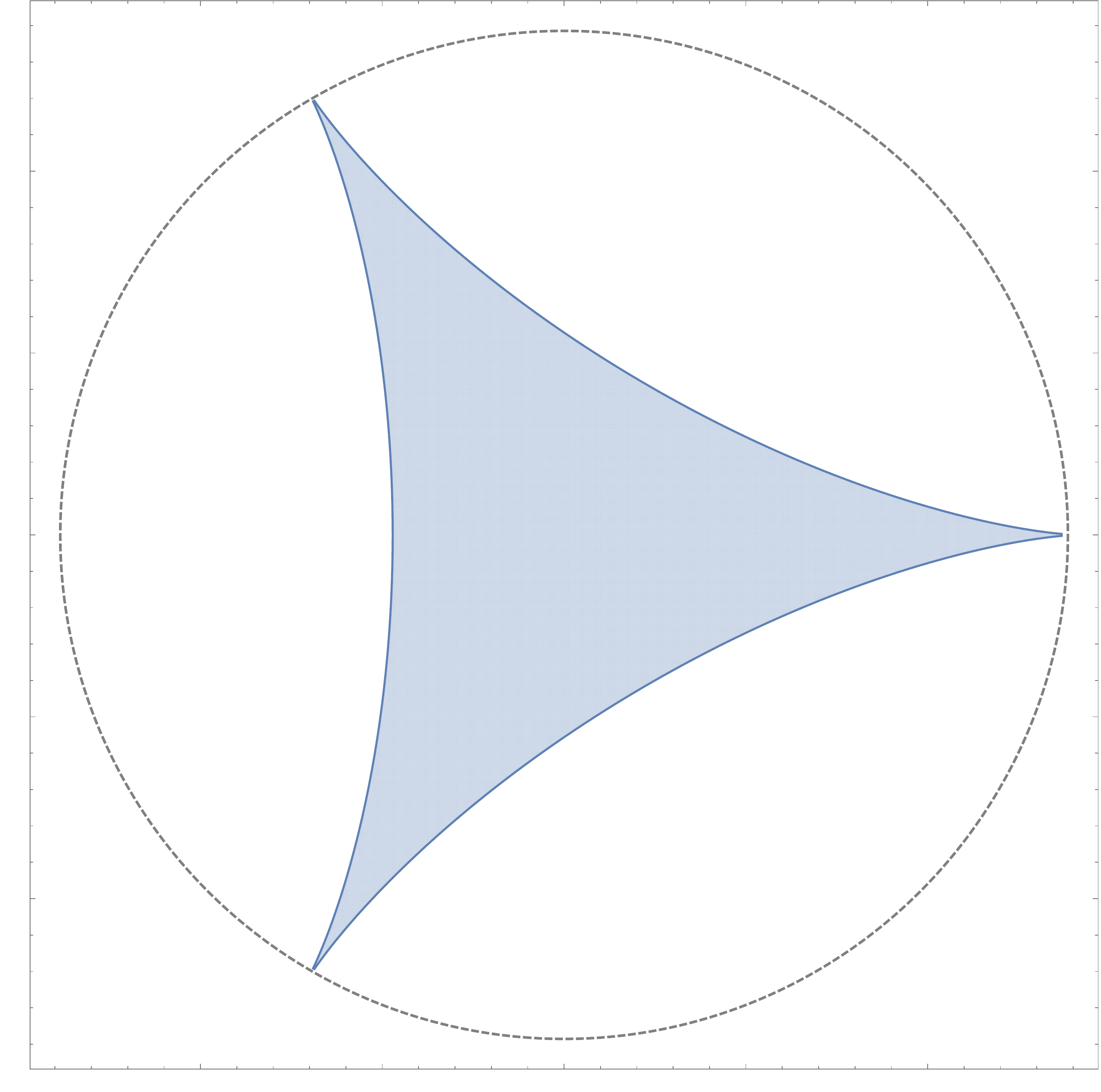}
\end{overpic}
\hspace{0.3cm}
\begin{overpic}[width=.55\textwidth]{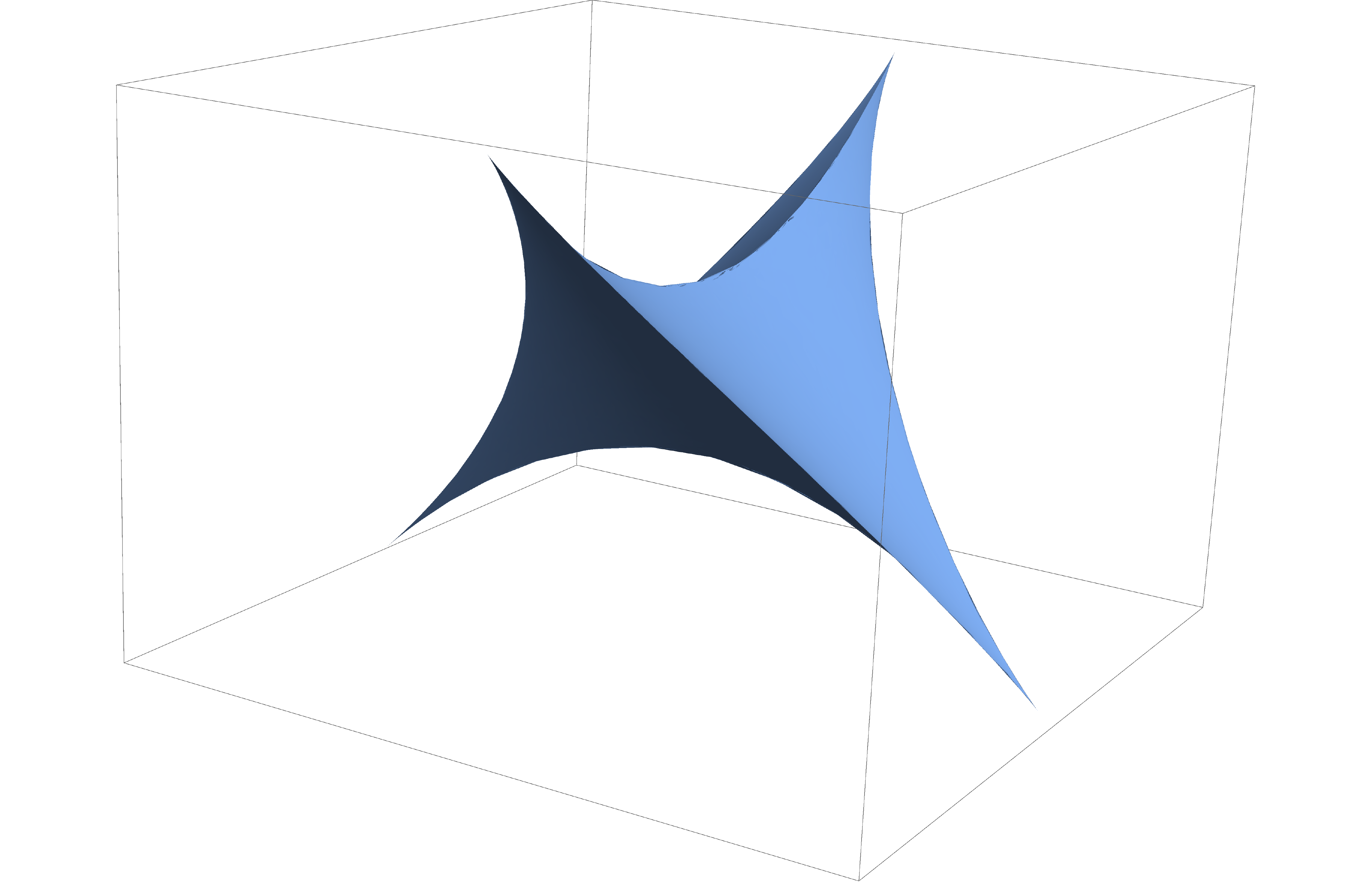}
\end{overpic}
\caption{The figure on the left corresponds to the orthogonality region for the case $n=2$. This is the area enclosed by Steiner's hypocycloid, 
which is given by an algebraic curve of fourth degree \eqref{eq: weight (2,1)}. The figure on the right is the three-dimensional region of orthogonality 
for $n=3$ which is determined by the algebraic equation of degree six \eqref{eq: weight (3,1)}.}
\label{fig:2d}
\end{figure}


\section{Inverting the branching rule}\label{sec:invertingbranchingrule}

The aim of this section is to calculate the set $P^{+}_{G}(k\omega_{1})$ for $k\in\bbN_{0}$, i.e.~the set of irreducible $G$-representations 
$\pi^{G}_{\lambda}$ such that $[\pi^{G}_{\lambda}|_{H}:\pi^{H}_{k\omega_{1}}]=1$. 
The pair $(G,H)$ is a spherical pair, meaning that a Borel subgroup of $G$ has an open orbit on the quotient $G/H$. 
The open orbit corresponds to the open Bruhat cell via the isomorphism $G/H \cong\SL(n+1,\bbC)$ which is induced from the map $G\to\SL(n+1,\bbC)$ 
$(g_1,g_2)\mapsto g_1g_2^{-1}$. 
In particular, this shows that by \cite[Thm.~25.1]{Timashev} the trivial representation occurs with multiplicity at most $1$ in 
$\pi^G_\lambda\vert_H$. 

Let $P\subset H$ denote the parabolic subgroup that contains the Borel subgroup of $H$ of upper triangular matrices and whose Levi subgroup 
has simple roots given by $\{\alpha_{2},\ldots, \alpha_{n}\}$. The fundamental weight $\omega_{1}$ extends to a character of $P$, and so does $k\omega_{1}$.  
%
%
Let $L\to G/P$ be a $G$-equivariant line bundle. Its space of global sections is a $G$-module. 
One can show that all such modules decompose multiplicity free into irreducible $G$-modules if and only if $P\subset G$ is a spherical subgroup, 
see e.g.~\cite[Thm.~25.1]{Timashev}. It turns out that for this choice of parabolic subgroup $P\subset G$ the pair $(G,P)$ is still spherical. 
The parabolic subgroup associated to $\{\alpha_{1},\ldots,\alpha_{n-1}\}$ also has this property, 
but there are essentially no other parabolic subgroups for which this holds, see \cite[\S6]{He et al}.

We explain how to describe the decomposition of the spaces of sections of all such associated line bundles at once.

\begin{definition}\label{def:extendedweightsemigroup}
Let $G'$ be a connected simply connected reductive group and let $G''\subset G'$ be a spherical subgroup, i.e.~the quotient $G'/G''$ 
admits an open orbit for the action of a Borel subgroup $B'\subset G'$. Let $T'\subset B'$ be a maximal torus. 
Denote by $X^{+}(T')$ the semi-group of positive characters of $T'$ with respect to $B'$ and by $X(G'')$ the group of characters of $G''$. 
For $\lambda \in X^{+}(T')$ and $\mu\in X(G'')$ put 
$$
\bbC[G']^{(B'\times G'')}_{(\lambda,\mu)} = \{f\colon G'\to\bbC\mid \forall(b,g,h)\in B'\times G'\times G'':\,f(b^{-1}gh)=\lambda(b)f(g)\mu(h) \}
$$
and define
$$
\widehat{\Lambda}_{+}(G',G'')=\{(\lambda,\mu)\in X^{+}(T')\times X(G'') \mid \bbC[G']^{(B'\times G'')}_{(\lambda,\mu)}=\bbC\},
$$
which is called the extended weight semi-group of the pair $(G',G'')$.  
\end{definition} 

Definition \ref{def:extendedweightsemigroup} follows \cite[Def.~1]{AvdeevGorfinkel}, since we have moreover assumed that $(G',G'')$ is 
a spherical pair, so that the dimension of $\bbC[G']^{(B'\times G'')}_{(\lambda,\mu)}$ is at most $1$, see \cite[Thm.~25.1]{Timashev}.
One can show that $\widehat{\Lambda}_{+}(G',G'')$ is a semi-group and moreover that it is freely generated, the generators corresponding to the set of 
$B'$-stable prime divisors on $G'/G''$, see \cite[Thm.~2]{AvdeevGorfinkel}.

Observe that $(\lambda^{*},k\omega_{1})\in\widehat{\Lambda}_{+}(G,P)$ if and only if 
$[\pi^{G}_{\lambda}|_{P}: k\omega_1]=1$, see \cite[\S 1.2]{AvdeevGorfinkel}, and this happens if and only 
if $[\pi^{G}_{\lambda}|_{H}: \pi^H_{k\omega_1}]=1$. 
Hence $P^{+}_{G}(k\omega_{1})$ consists of elements $\lambda\in P^{+}_{G}$ such that $(\lambda^{*},k\omega_{1})\in\widehat{\Lambda}_{+}(G,P)$. 
We calculate $\widehat{\Lambda}_{+}(G,P)$ in Lemma \ref{lemma: extended weight semi group}.

In this subsection we use a different choice of positive roots for $G$, namely the one that corresponds to the Borel subgroup $B\times B$, where $B\subset\SL(n+1,\bbC)$ consists of upper triangular matrices.

This new choice of positivity is related to our earlier choice by applying the longest Weyl group element of the second factor to the second component. The fundamental weights are now given by $(\omega_{i},0), (0,\omega_{j})$ and the fundamental spherical weights are given by $\eta_{i}=(\omega_{i},\omega_{n+1-i})$.  Furthermore we employ the convention $\omega_{0}=\omega_{n+1}=0$.

\begin{lemma}\label{lemma: extended weight semi group}
The extended weight semi-group $\widehat{\Lambda}_{+}(G,P)$ is generated by
\begin{equation}\label{eqn: generators}
((\omega_{i},\omega_{n+1-i})^{*},0),\quad i=1,\ldots,n\quad\mbox{and }((\omega_{i},\omega_{n+2-i})^{*},\omega_{1}),\quad i=1,\ldots,n+1.
\end{equation}
\end{lemma}

\begin{proof}
The elements $((\omega_{i},\omega_{n+1-i})^{*},0), i=1,\ldots,n$ correspond to spherical representations and are thus contained in $\widehat{\Lambda}_{+}(G,P)$. To show that the elements $((\omega_{i},\omega_{n+2-i})^{*},\omega_{1}),i=1,\ldots,n+1$ are contained in $\widehat{\Lambda}_{+}(G,P)$ we have to show that the irreducible $G$-representation $V^{G}_{(\omega_{i},\omega_{n+2-i})}$ contains $V^{H}_{\omega_{1}}$ upon restriction to the diagonal subgroup $H$. This can be done by means of the Littlewood-Richardson rule, see e.g.~\cite[\S9.3.5]{Goodman and Wallach}. Instead of giving this argument we refer to Corollary \ref{corollary: concrete embedding} where we calculate the corresponding embeddings.

The elements in (\ref{eqn: generators}) are indecomposable and linearly independent. To prove the result it suffices to show that the rank of $\widehat{\Lambda}_{+}(G,P)$ is at most $2n+1$.

Consider the fibration $G/P\to G/H$. On $G/H$ the number of $B\times B$-stable prime divisors is $n$, which follows for example from the Bruhat decomposition. The pull-back of each of these divisors gives a $B\times B$-stable prime divisor on $G/P$. The other $B\times B$-stable prime divisors in $G/P$ map dominantly onto $G/H$. This means that these divisors intersect the fiber $H/P$ in a $B_{M}$-stable prime divisor where $B_{M}=(B\times B)\cap H\subset M\cong(\bbC^{\times})^{n}$ is a torus that acts naturally on $H/P\cong \bbP^{n}(\bbC)$. There are $n+1$ prime divisors in $H/P$ that are stable under $M$, namely the hyperplanes $\{(z_{0}:\ldots:z_{n})\in\bbP^{n}(\bbC) \mid z_{i}=0\}$ for $i=0,\ldots,n$. This shows that there are at most $2n+1$ different $B$-stable prime divisors in $G/P$, as desired.
\end{proof}

\begin{corollary}\label{cor: well is product}
Fix $k\in\bbN_{0}$ and set $B(k\omega_{1})=\{(\sum_{i=1}^{n+1}k_{i}(\omega_{i},\omega_{n+2-i}):\sum_{i=1}^{n+1}k_{i}=k\}$. Then $P^{+}_{G}(k\omega_{1})=B(k\omega_{1})+P_{G}^{+}(0)$.
\end{corollary}

\begin{proof}
Note that $\lambda\in P^{+}_{G}(k\omega_{1})$ if and only if $(\lambda^{*},k\omega_{1})\in\widehat{\Lambda}_{+}(G,P)$, which is in turn equivalent to
$$\lambda=\sum_{i=1}^{n+1}k_{i}(\omega_{i},\omega_{n+2-i})+\sum_{j=1}^{n}d_{j}(\omega_{i},\omega_{n+1-i}),\quad\mbox{with }\sum_{i=1}^{n+1}k_{i}=k.$$
This settles the claim.
\end{proof}

We proceed to check how $P^{+}_{G}(k\omega_{1})$ behaves with respect to the tensor product. Define $\beta_{i}=(\omega_{i}-\omega_{i+1},\omega_{n+2-i}-\omega_{n+1-i})$.
Then $B(k\omega_{1})$ is contained in the affine plane that is parallel to $\mathrm{span}(\beta_{1},\ldots,\beta_{n})$.
Recall that the fundamental spherical weights with respect to the Borel subgroup $B\times B$ are given by $\eta_{i}=(\omega_{i},\omega_{n+1-i})$. A basis of 
$\lat_{G}^{*}$ is given by $(\beta_{1},\ldots,\beta_{n},\eta_{1},\eta_{n})$.
Observe that 
\begin{itemize}
\item $(\alpha_{1},0)=\beta_{1}+\eta_{1}$,
\item $(\alpha_{i},0)=\beta_{i}+\eta_{i}-\eta_{i-1}$, for $i=2,\ldots,n$,
\item $(0,\alpha_{i})=-\beta_{i}-\eta_{i}+\eta_{i+1}$, for $i=1,\ldots,n-1$,
\item $(0,\alpha_{1})=-\beta_{n}+\eta_{n}$.
\end{itemize}
Any weight that occurs in the decomposition of the tensor product $V^{G}_{\lambda}\otimes V^{G}_{\eta_{i}}$ is of the form $\lambda+\eta_{i}-\sum_{\alpha>0}(n_{(\alpha,0)}(\alpha,0)+n_{(0,\alpha)}(0,\alpha))$ for some coefficients $n_{(\alpha,0)},n_{(0,\alpha)}\in\bbN_{0}$ and is hence of degree $\le|\lambda|+1$.

The dominant weight $(\omega_{i},\omega_{n+2-i})$ corresponds to the dominant weight $(\omega_{i},-\omega_{i-1})$ with respect to the Borel subgroup $B\times B^{-}$, where $B^{-}$ is opposite to $B$. Restricting this dominant weight to $\lat_{M}$ gives $\frac{1}{2}(\omega_{i}-\omega_{i-1},\omega_{i}-\omega_{i-1})$. This element corresponds to the weight vector $\omega_{i}-\omega_{i-1}$ on $\lat$. The map
$$B(k\omega_{1})\to P^{+}_{M}(k\omega_{1}):\sum_{i=1}^{n+1}k_{i}(\omega_{i},\omega_{n+2-i})\mapsto \sum_{i=1}^{n+1}k_{i}(\omega_{i}-\omega_{i-1})$$
is surjective, which is a general feature for multiplicity free systems, see e.g.~\cite[Thm.3.1]{mvp-MVMVOP}. To see that it is injective, we have to understand the branching $\pi^{H}_{k\omega_{1}}|_{T_{H}}$. The weight vectors are just the monomials $\prod_{i=1}^{n+1}e_{i}^{k_{i}}$ and their weights are $\sum_{i=1}^{n}(k_{i}-k_{i+1})\omega_{i}=\sum_{i=1}^{n+1}k_{i}(\omega_{i}-\omega_{i-1})$.
We observe that projection along the spherical directions $\eta_{1},\ldots,\eta_{n}$ provides a bijection $B(k\omega_{1})\to P^{+}_{M}(k\omega_{1})$.

We have shown that Conditions \ref{cond:multfree}, \ref{cond:structurePbottom} and \ref{cond:tensor} are satisfied.

\begin{remark}\label{rk: Weyl group action}
The Weyl group $W(\Sigma)=S_{n+1}$ acts transitively on $P^{+}_{M}(\omega_{1})$. Indeed, the standard basis of $V$ consists of $T$-weight vectors $e_{1},\ldots,e_{n+1}$ and $T$ acts with the characters $\xi_{i}:T\to\bbC^{\times}:t\mapsto t_{i}$. We have $w(\xi_{i})(t)=\xi_{i}(w^{-1}t)=t_{w(i)}=\xi_{w(i)}(t)$, which shows that the action of $W(\Sigma)$ on $P^{+}_{T}(\omega_{1})$ is basically the same as the action of $S_{n+1}$ on the set $\{1,\ldots,n+1\}$ and is thus transitive.
\end{remark}



\section{The matrix weight}\label{sec:matrixweightinpart2}

\subsection{Some representations}

We discuss some representations of $G$ and $H$ that are needed to calculate the spherical functions of degree zero. Note that $V^{H}_{k\omega_{1}}=S^{k}(V)$, the $k$-th symmetric power $V$. We identify $V=\bbC^{n+1}$ with its standard basis $(e_{1},\ldots,e_{n+1})$. 
A basis of $S^{k}(V)$ is given by the monomials $e^{\tau}=e_{1}^{\tau_{i}}\cdots e_{n+1}^{\tau_{n+1}}$, where $\tau\in\bbN_{0}^{n+1}$ is a composition of $k$ in at most $n+1$ parts, i.e.~$\sum_{i=1}^{n+1}\tau_{i}=k$. For such a composition we introduce the binomial $\binom{k}{\tau}=k!/(\tau_{1}!\cdots\tau_{n+1}!)$. We identify $P^{+}_{M}(k\omega_{1})$ with the set of compositions $\tau\in\bbN_{0}^{n+1}$ of $k$. The element in $P^{+}_{G}(k\omega_{1})$ whose projection onto $B(\mu)$ along the spherical directions is $\sigma$ is denoted by $\lambda(d,\sigma)$, where $d\in\bbN_{0}^{n}$ is the degree. More precisely $\lambda(\sigma,d)=\sigma+\sum_{i=1}^{n}d_{j}(\omega_{i},\omega_{n+1-i})$ following Corollary \ref{cor: well is product}.

\begin{lemma}\label{lemma: inner product on symmetric power}
The inner product on $V^{H}_{k\omega_{1}}=S^{k}(V)$ with $||e_{\sigma}||^{2}=\binom{k}{\sigma}^{-1}$ is $H_{c}$ invariant.
\end{lemma}

\begin{proof}
Consider the $H$-equivariant embedding $\iota:S^{k}(V)\to V^{\otimes k}:\binom{k}{\sigma}e_{\sigma}\mapsto\sum_{w\in S_{k}^{I_{\sigma}}}e_{w(1)}\otimes\cdots\otimes e_{w(n+1)}$, 
where $S_{k}^{I_{\sigma}}$ denotes the set of unique representatives of smallest length of the cosets $S_{\tau_{i}}/(S_{s_{i}}\times\cdots\times S_{s_{n+1}})$. The latter has a natural $H_{c}$-invariant Hermitian inner product. We stipulate that $\iota$ is isometric, which implies $\binom{k}{\sigma}^{2}||e_{\sigma}||^{2}=\binom{k}{\sigma}$ and the result follows.
\end{proof}

We refer to this inner product on $S^{k}(V)$ as the standard inner product. The inner product on $\bigotimes_{i=1}^{n+1}S^{\tau_{i}}(V)$ that is given by the product of the inner products is also referred to as the standard inner product. Define
$$M(\tau,\rho)=\left\{(s^{1},\ldots,s^{n+1})\in\left(\bbN_{0}^{n+1}\right)^{n+1}\left|\forall p:\,
\sum_{q=1}^{n+1}s^{p}_{q}=\tau_{p},\forall q:\,\sum_{p=1}^{n+1}s^{p}_{q}=\rho_{q}
\right.\right\}.$$
An element of $M(\tau,\rho)$ is denoted by $(s)$, it is really an $(n+1)\times(n+1)$-matrix whose entries of the $p$-th column and $q$-th row add up to $\tau_{p}$ and $\rho_{q}$ respectively.

\begin{lemma}\label{lemma: CG embedding} A composition $\tau$ gives rise to an isometric $H$-equivariant embedding
$$
i_{\tau}:S^{k}(V)\to\bigotimes_{i=1}^{n+1}S^{\tau_{i}}(V):e_{\rho}\mapsto\binom{k}{\rho}^{-1}\sum_{(s)\in M(\tau,\rho)}\left(
\binom{\tau_{1}}{s^{1}}e^{s^{1}}\otimes\cdots\otimes \binom{\tau_{n+1}}{s^{n+1}}e^{s^{n+1}}
\right).$$
\end{lemma}

\begin{remark}
(i) Note that $i_{\tau}$ is easily defined on the highest weight vector. However, we need to have all the information of the Lemma \ref{lemma: CG embedding} for later purposes.

(ii) For $n=1$, Lemma \ref{lemma: CG embedding} provides the Clebsch-Gordan coefficients for the embeddings $H^{\ell}\to H^{\ell_{1}}\otimes H^{\ell_{2}}$ with $\ell_{1}+\ell_{2}=\ell$, see e.g.~\cite[Prop.2.1]{Koor1}. 
We have not tried to obtain the general Clebsch-Gordan coefficients since we do not require the explicit knowledge. Moreover, in general this seems to be a hard problem. 

(iii) The isometry property of $i_{\tau}$ gives the generalized Vandermonde summation. 
\end{remark}

\begin{proof}
Let $\alpha_{i}$ be a simple root and consider the root vector $E_{i}\in\lag_{\alpha_{i}}$, which acts on $S^{k}(V)$ by $e_{i}\frac{d}{de_{i+1}}$ by identifying $S^{k}(V)$ with the space of homogeneous polynomials of degree $k$ on $V^{*}$. 
Given a composition $\rho=(\rho_{1},\ldots,\rho_{n+1})$ of $k$, let $\rho(i)$ denote the composition
$$\rho(i)=(\rho_{1},\ldots,\rho_{i}+1,\rho_{i+1}-1,\ldots,\rho_{n+1}).$$
We allow a negative number in the composition, in which case we employ the convention that the binomial for such a composition is zero. 
We use the formula $\rho_{i+1}\binom{k}{\rho}=(\rho_{i}+1)\binom{k}{\rho(i)}$ to derive
\begin{multline}\label{eqn: comb intertwiner}
E_{i}i_{\tau}(e_{\rho})=\\
\binom{k}{\rho}^{-1}\sum_{(s)\in M(\tau,\rho)}\sum_{k=1}^{n+1}\left(
\binom{\tau_{1}}{s^{1}}e^{s^{1}}\otimes\cdots\otimes(s_{i}+1)\binom{\tau_{k}}{s^{k}(i)}\frac{e_{i}e^{s^{k}}}{e_{i+1}}\otimes\cdots\otimes \binom{\tau_{n+1}}{s^{n+1}}e^{s^{n+1}}
\right).
\end{multline}
Observe that we obtain a linear combination of elements of the form $e^{\sigma^{1}}\otimes\ldots\otimes e^{\sigma^{n+1}}$ with $\sigma\in M(\tau,\rho(i))$. Let $s(i,k)=(s^{1},\ldots,s^{k}(i),\ldots,s^{n+1})$ and note that every $\sigma\in M(\tau,\rho(i))$ is of the form $s(i,k)$, for some $k\in\{1,\ldots,n+1\}$ and $s^{k}_{i+1}>0$. Indeed, if $\sigma\in M(\tau,\rho(i))$ and $\sigma^{k}_{i}\ne0$ then we define $s(\sigma,k)\in M(\tau,\rho)$ by $s(\sigma,k)^{\ell}=\sigma^{\ell}$ if $\ell\ne k$ and
$$s(\sigma,k)^{k}=(\sigma^{k}_{1},\ldots,\sigma^{k}_{i}-1,\sigma^{k}_{i+1}+1,\ldots,\sigma^{k}_{n+1}).$$
One checks that $s(\sigma,k)(i,k)=\sigma$. If $\sigma^{k}_{i}=0$ for all $k=1,\ldots,n+1$, then $\sum_{k}\sigma^{k}_{i}=0$, but this sum is also equal to $\rho_{i}+1$, and this contradicts $\rho_{i}\in\bbN_{0}$. We use this observation to rewrite \eqref{eqn: comb intertwiner},
\begin{multline*}
E_{i}i_{\tau}(e_{\rho})=\binom{k}{\rho}^{-1}\sum_{(\sigma)\in M(\tau,\rho(i))}\sum_{k=1}^{n+1}\sigma^{k}_{i}\left(
\binom{\tau_{1}}{\sigma^{1}}e^{\sigma^{1}}\otimes\cdots\otimes \binom{\tau_{n+1}}{\sigma^{n+1}}e^{\sigma^{n+1}}
\right)\\
=\binom{k}{\rho}^{-1}(\rho_{i}+1)\sum_{(\sigma)\in M(\tau,\rho(i))}\left(
\binom{\tau_{1}}{\sigma^{1}}e^{\sigma^{1}}\otimes\cdots\otimes \binom{\tau_{n+1}}{\sigma^{n+1}}e^{\sigma^{n+1}}
\right)=\rho_{i+1}i_{\tau}(E_{i}e_{\rho}),
\end{multline*}
as desired. We have shown that actions of the root vectors of the simple positive roots are intertwined by $i_{\tau}$. In a similar fashion one checks that $i_{\tau}$ intertwines the action of the root vectors of negative roots and of the torus. Finally note that $||i_{\tau}(e_{1}^{k})||=||e_{1}^{\tau_{1}}\otimes\ldots\otimes e^{\tau_{n+1}}_{1}||=1$, which implies that $i_{\tau}$ is an isometry. 
\end{proof}


\subsection{Calculation of $\Phi_{0}^{k\omega_{1}}$}

Let $\mu=k\omega_{1}$. Consider the spherical functions $\{\Phi^{\mu}_{\lambda(0,\sigma)} \mid \sigma\in P^{+}_{M}(\mu)\}$. 
Following the proof of \cite[Lem.~6.1]{mvp-MVMVOP}, $\Phi_{a}^{\mu}=(\Phi^{\mu}_{\lambda(0,\sigma)}(a) \mid \sigma\in P^{+}_{M}(\mu))$ is a basis of 
$\End_{M}(V^{H}_{\mu})$ for $a\in A_{\mu-\reg}$. 
By Schur's Lemma, another basis of $\End_{M}(V^{H}_{\mu})$ is given by $\calF\otimes\calE =(f_{\sigma}\otimes e_{\sigma}\mid \sigma\in P^{+}_{M}(\mu))$, 
where $\calE=(e_{\sigma}\mid \sigma\in P^{+}_{M}(\mu))$ and $\calF=(f_{\sigma} \mid \sigma\in P^{+}_{M}(\mu))$ 
the basis of $(S^{k}(\bbC^{n+1}))^{*}$ dual to $\calE$. The base change yields the full spherical function of degree zero,
$$\Phi_{0}^{\mu}(a)=[\mathrm{I}]^{\Phi_{a}^{\mu}}_{\calF\otimes\calE}=\left(\frac{\langle e_{\sigma},a\cdot e_{\sigma}\rangle_{\lambda(0,\tau)}}{\langle e_{\sigma},e_{\sigma}\rangle_{\lambda(0,\tau)}}\right)_{\sigma,\tau}\in\End(\bbC^{n+1}).$$
This matrix is in general hard to compute. However, for the case $(\SU(2)\times\SU(2),\diag(\SU(2)))$ there exists a remarkable formula found by Koornwinder, \cite[Prop.~3.2]{Koor1}. We found a similar formula for the matrix $\Phi^{\mu}_{0}(a)$, whose formulation and proof occupies the rest of this subsection.

Let $\tau=(\tau_{1},\ldots,\tau_{n+1})\in P^{+}_{M}(\mu)$ and consider the standard $G$-representation $\pi^{G}_{T(\tau)}$ on $T(\tau)=\bigotimes_{i=1}^{n+1}(V^{G}_{\lambda_{i}})^{\otimes\tau_{i}}$. Let $\Gamma=(\gamma_{\tau}|\tau\in P^{+}_{M}(\mu))$ be a collection of $H$-equivariant isometric embeddings $\gamma_{\tau}:V^{H}_{\mu}\to T(\tau)$ and let $\gamma_{\tau}^{*}:T(\tau)\to V^{H}_{\mu}$ denote their adjoint maps. Define
$$\Gamma^{\mu}_{\tau}(a)=\gamma_{\tau}^{*}\circ\pi^{G}_{T(\tau)}(a)\circ\gamma_{\tau}$$
and observe that $\Gamma^{\mu}_{\tau}(a)=\sum_{\lambda'\le\lambda(0,\tau)}c_{\lambda',\gamma_{\tau}}\Phi^{\mu}_{\lambda'}(a)$. Moreover, the coefficients $c_{\lambda',\gamma_{\tau}}$ are non-negative numbers that add up to one. Define $C(\Gamma)\in\End(\bbC^{N})$ by
$$C(\Gamma)_{\sigma,\tau}=c_{\lambda(0,\sigma),\gamma_{\tau}}.$$
Consider the map $\Gamma^{\mu}_{a}:\End_{M}(V^{H}_{\mu})\to\End_{M}(V^{H}_{\mu}):f_{\tau}\otimes e_{\tau}\mapsto\Gamma^{\mu}_{\tau}(a)$. Its matrix with respect to 
the basis $\calF\otimes\calE$ is given by
\begin{eqnarray}\label{eqn: matrices Gamma Phi C}
[\Gamma^{\mu}_{a}]^{\calF\otimes\calE}_{\calF\otimes\calE}=\Phi^{\mu}_{0}(a)\cdot C(\Gamma).
\end{eqnarray}
We proceed to calculate this matrix for a specific collection $\Gamma$.
\begin{definition}
Given $a\in A$, define $g_{a}\in\End(\bbC^{n+1})$ by $(g_{a})_{ij}=\langle a\cdot e_{i},e_{i}\rangle_{\lambda_{j}}$.
\end{definition}
In fact, $g_{a}=\Phi^{\omega_{1}}_{0}(a)\in\End(\bbC^{n+1})$, since the basis $(e_{1},\ldots,e_{n+1})$ is orthonormal with respect to the $H$-invariant inner product on $V^{H}_{\omega_{1}}$. Moreover, $g_{a}$ is invertible for $a\in A_{\mu-\reg}$.

\begin{lemma}\label{lemma: g_a acting}
The matrix of the natural action of $g_{a}$ on $S^{k}(V^{H}_{\omega_{1}})$ is given by
$$\left([g_{a}]^{\calE}_{\calE}\right)_{\rho,\tau}=
\sum_{(s^{1},\ldots,s^{n+1})\in M(\rho,\tau)}\left(
\prod_{i=1}^{n+1}\binom{\tau_{i}}{s^{i}}\prod_{j=1}^{n+1}\langle a\cdot e_{j},e_{j}\rangle_{\lambda_{i}}^{s^{i}_{j}}
\right).$$
\end{lemma}

\begin{proof} Let $S(\tau_{i})=\{s\in\bbN_{0}^{n+1}|\,\sum_{j=1}^{n+1}s_{j}=\tau_{i}\}$. The calculation
\begin{multline*} 
g_{a}e_{\tau}=(g_{a}e_{1})^{\tau_{1}}\cdots(g_{a}e_{n+1})^{\tau_{n+1}}=\prod_{i=1}^{n+1}\left(\sum_{j=1}^{n+1}\langle a\cdot e_{j},e_{j}\rangle_{\lambda_{i}}e_{j}\right)^{\tau_{i}}=\\
\prod_{i=1}^{n+1}\left(
\sum_{s\in S(\tau_{i})}\binom{\tau_{i}}{s}\prod_{j=1}^{n+1}\langle a\cdot e_{j},e_{j}\rangle_{\lambda_{i}}^{s_{j}}e_{j}^{s_{j}}
\right)=\\
\sum_{\rho}\left(
\sum_{(s^{1},\ldots,s^{n+1})\in M(\rho,\tau)}\left(
\prod_{i=1}^{n+1}\binom{\tau_{i}}{s^{i}}\prod_{j=1}^{n+1}\langle a\cdot e_{j},e_{j}\rangle_{\lambda_{i}}^{s^{i}_{j}}
\right)
\right)e_{\rho}
\end{multline*}
implies the claim.
\end{proof}

The coefficient of $e_{\rho}$ can be interpreted as follows. According to Lemma \ref{lemma: CG embedding}, the composition $\tau$ gives rise to the $H$-equivariant isometric embedding
\begin{displaymath}
S^{k}(V)\to\bigotimes_{i=1}^{n+1}S^{\tau_{i}}(V):e_{\rho}\mapsto \binom{k}{\rho}^{-1}\sum_{(s)\in M(\tau,\rho)}\left(
\binom{\tau_{1}}{s^{1}}e^{s^{1}}\otimes\cdots\otimes \binom{\tau_{n+1}}{s^{n+1}}e^{s^{n+1}}
\right).
\end{displaymath}
Each of the tensor factors embeds $H$-equivariantly isometrically into the corresponding tensor power,
$$as_{\tau_{i}}:S^{\tau_{i}}(V)\to V^{\otimes\tau_{i}}:\binom{\tau_{i}}{s}e_{s}\mapsto \sum_{w\in S_{\tau_{i}}^{I_{s}}}e_{w(1)}\otimes\cdots\otimes e_{w(\tau_{i})},$$
where $S_{\tau_{i}}^{I_{s}}$ is as in the proof of Lemma \ref{lemma: inner product on symmetric power}. Note that $|S_{\tau_{i}}^{I_{s}}|=\binom{\tau_{i}}{s}$. In turn, the $H$-equivariant isometric embedding $\beta^{\omega_{1}}_{\lambda_{i}}:V\to V^{G}_{\lambda_{i}}$ induces an $H$-equivariant embedding of the tensor powers,
$$(\beta^{\omega_{1}}_{\lambda_{i}})^{\otimes\tau_{i}}:V^{\otimes\tau_{i}}\to(V^{G}_{\lambda_{i}})^{\otimes\tau_{i}}.$$
Denote $c_{\tau_{i}}=(\beta^{\omega_{1}}_{\lambda_{i}})^{\otimes\tau_{i}}\circ as_{\tau_{i}}$. We obtain the $H$-equivariant isometric embedding
\begin{equation}\label{eqn: gamma_tau}
\gamma_{\tau}:S^{k}(V)\to\bigotimes_{i=1}^{n+1}(V^{G}_{\lambda_{i}})^{\otimes\tau_{i}}:e_{\rho}\mapsto\binom{k}{\rho}^{-1}\sum_{(s)\in M(\tau,\rho)}\left(c_{\tau_{1}}(e^{s^{1}})\otimes\cdots\otimes c_{\tau_{n+1}}(e^{s^{n+1}})\right).
\end{equation}

\begin{lemma}\label{lemma: Gamma g_a}
We have
\begin{eqnarray*}
\langle\gamma_{\tau}(e_{\rho}),a\cdot\gamma_{\tau}(e_{\rho})\rangle=
\binom{k}{\rho}^{-2}\sum_{(s)\in M(\tau,\rho)}\left(
\prod_{i=1}^{n+1}\binom{\tau_{i}}{s^{i}}\prod_{j=1}^{n+1}\langle e_{j}, a\cdot e_{j}\rangle_{\lambda_{i}}^{s^{i}_{j}}
\right).
\end{eqnarray*}
\end{lemma}

\begin{proof}
The summands of $a\cdot\gamma(\tau)(e_{\rho})$ are weight vectors of $M$ whose weight is determined by $(s)\in M(\tau,\rho)$.
This implies
\begin{multline}\nonumber
\langle\gamma(\tau)(e_{\rho}),a\cdot\gamma(\tau)(e_{\rho})\rangle=\\
\binom{k}{\rho}^{-2}\sum_{(s)\in M(\tau,\rho)}\left\langle
c_{\tau_{1}}(e^{s^{1}})\otimes\cdots\otimes c_{\tau_{n+1}}(e^{s^{n+1}})
,
a\cdot c_{\tau_{1}}(e^{s^{1}})\otimes\cdots\otimes a\cdot c_{\tau_{n+1}}(e^{s^{n+1}})
\right\rangle.
\end{multline}
Finally we use
$$\langle c_{\tau_{i}}(e^{s}),a\cdot c_{\tau_{i}}(e^{s})\rangle=\binom{\tau_{i}}{s}\prod_{j=1}^{n+1}\langle e_{j},a\cdot e_{j}\rangle_{\lambda_{i}}^{s_{j}},$$
which finishes the proof.
\end{proof}

Let $\Gamma=(\gamma_{\tau}\mid\tau\in P^{+}_{M}(\mu))$ where the $\gamma_{\tau}$ are given by \eqref{eqn: gamma_tau}. 
Let $\calE_{n}$ denote the normalized basis $(\binom{k}{\sigma}^{1/2}e_{\sigma}\mid\,e_{\sigma}\in\calE)$.

\begin{theorem}
\label{thm:expression_Phi0}
Let $a\in A_{\reg}$ and consider $g_{a}\in\GL_{n+1}(\bbC)$. Let $D\in\End(\bbC^{N})$ be the diagonal matrix with entries $D_{\sigma,\sigma}=||e_{\sigma}||=\binom{k}{\sigma}^{-1/2}$. Then
\begin{eqnarray*}
\Phi^{\mu}_{0}(a)\cdot C(\Gamma)=D\cdot[g_{a}]^{\calE_{n}}_{\calE_{n}}\cdot D\in\End(\bbC^{N}).
\end{eqnarray*}
\end{theorem}

\begin{proof}
Lemma \ref{lemma: g_a acting} and Lemma \ref{lemma: Gamma g_a} imply that $D^{2}\cdot[g_{a}]^{\calE}_{\calE}=[\Gamma_{a}^{\mu}]^{\calF\otimes\calE}_{\calF\otimes\calE}$. Following \eqref{eqn: matrices Gamma Phi C} we find $D^{2}\cdot[g_{a}]^{\calE}_{\calE}=\Phi^{\mu}_{0}(a)\cdot C(\Gamma)$. The base change $[\mathrm{I}]^{\calE}_{\calE_{n}}=D$ implies the result.
\end{proof}

\begin{corollary}
$\det(C(\Gamma))\ne0$.
\end{corollary}

\begin{remark}
For $n=1$ we know that $\lambda\in B(\mu)$ implies $\lambda-\alpha\not\in B(\mu)$. This implies that $C(\Gamma)=\mathrm{I}$. We obtain a new proof of \cite[Prop.~3.2]{Koor1}.
\end{remark}

\begin{remark}
The decomposition of $T(\tau)$ into irreducible $G$-representations seems to be a challenging problem. But in fact, this decomposition is not enough to give the matrix $C(\Gamma)$. Indeed, the matrix $C(\Gamma)$ describes the embeddings $\gamma_{\tau}\in\Gamma$.
\end{remark}


\subsection{The element $g_{a}$}

We proceed to calculate the element $g_{a}$ in the general case. To this end, we need the embeddings $V^{H}_{\omega_{1}}\to V^{G}_{\lambda_{i}}$ and the projections $V^{G}_{\lambda_{i}}\to V^{H}_{\omega_{1}}$. 

Let $\calJ_{i}$ denote the set of $i$-tuples $1\le j_{1}<\cdots<j_{i}\le n+1$. For $\kappa=1,\ldots,i$ and $J\in\calJ_{i}$ we denote by $J(\kappa)$ the $i-1$-tuple that we obtain from $J$ by omitting $j_{\kappa}$. 

Let $(e_{1},\ldots,e_{n+1})$ denote the standard basis of $V$. Then $(e_{J}=e_{j_{1}}\wedge\ldots\wedge e_{j_{i}} \mid J\in\calJ_{i})$ is a basis of $\bigwedge^{i}V$. Let $\iota:\bigwedge^{n+1}V\to\bbC$ be the isomorphism defined by $\iota(e_{1}\wedge\ldots\wedge e_{n+1})=1$. Given $J\in\calJ_{i}, J'\in\calJ_{n+1-i}$ we denote $\eps(J,J')=\iota(e_{J}\wedge e_{J'})$.

\begin{lemma}
The irreducible $\SL(n+1,\bbC)\times\SL(n+1,\bbC)$-representations
$$\left(\bigwedge^{i}V\right)\otimes\left(\bigwedge^{n+2-i}V\right),\quad i=1,\ldots, n+1,$$
contain $V$ upon restriction to the diagonal. The embedding is given on the highest weight vector by $e_{1}\mapsto\sum\eps(J,K(1))e_{J}\otimes e_{K}$, where we sum over the $J\in\calJ_{i},K\in\calJ_{n+2-i}$ with $J\cap K=\{1\}$.
\end{lemma}

\begin{proof}
Note that the multiplicity is at most one. We start by finding a basis of the weight space of $\left(\bigwedge^{i}V\right)\otimes\left(\bigwedge^{n+2-i}V\right)$ for $M=\diag(T)$ of weight $\omega_{1}$. 
This space has a basis of weight vectors for $T\times T$. Certainly it contains the vectors $e_{J}\otimes e_{K}$ with $J\in\calJ_{i}$ and $K\in\calJ_{n+2-i}$ for which $J\cap K=\{1\}$. In fact, these vectors span the weight space under consideration. Indeed, let $e_{J}\otimes e_{K}$ be a weight vector of weight $\omega_{1}$. Then either $J$ or $K$ contains $1$, say $1\in K$. Then we must have $J\cup(K\backslash\{1\})=\{1,\ldots,n+1\}$, which implies $J\cap K=\{1\}$.

Now we show that the root vectors of $\SL(n+1,\bbC)$ of the positive simple roots annihilate a non-zero vector of the weight space $\mathrm{span}\{e_{J}\otimes e_{K} \mid J\in\calJ_{i},K\in\calJ_{n+2-i}, J\cap K=\{1\}\}$. We have $E_{\alpha_{k}}(e_{J}\otimes e_{K})\ne0$ if and only if $k\in J, k+1\in K$ or $k\in K, k+1\in J$. Indeed, $E_{\alpha_{k}}(e_{J}\otimes e_{K})=(E_{\alpha_{k}}e_{J})\otimes e_{K}+e_{J}\otimes(E_{\alpha_{k}}e_{K})$ and this is zero if $k$ and $k+1$ are in the same set $J$ or $K$. From this we deduce that
$$\sum\eps(J,K(1))e_{J}\otimes e_{K},$$
where we sum over the $J\in\calJ_{i},K\in\calJ_{n+2-i}$ with $J\cap K=\{1\}$, is annihilated by the root vectors $E_{\alpha_{k}}$, $k=1,\ldots,n$. This is clear for $k=1$, so we assume $k>1$. Whenever $k\in J$ and $k+1\in K$, then $J'=s_{k,k+1}J,K'=s_{k,k+1}K$  has $k+1\in J',k\in K'$ and $\eps(J,K(1))=-\eps(J',K'(1))$. However, $E_{\alpha_{k}}(e_{J}\otimes e_{K})=E_{\alpha_{k}}(e_{J'}\otimes e_{K'})$. This establishes the claim.
\end{proof}

\begin{lemma}
The $H$-equivariant projections $p_{i}:V^{G}_{(\omega_{i},\omega_{n+2-i})}=\bigwedge^{i}V\otimes\bigwedge^{n+2-i}V\to V$ are given by
\begin{equation}\label{eqn: projection}
e_{J}\otimes e_{K}\mapsto\sum_{\kappa=1}^{n+2-i}(-1)^{\kappa-1}\iota(e_{J}\wedge e_{K(\kappa)})e_{k_{\kappa}}.
\end{equation}
\end{lemma}

\begin{proof}
Consider the multi-linear map $\widetilde{p}_{i}:V^{n+2}\to V$ given by
$$(v_{j_{1}},\ldots,v_{j_{i}},w_{k_{1}},\ldots,w_{k_{n+2-i}})\mapsto\sum_{\kappa=1}^{n+2-i}(-1)^{\kappa-1}\iota(v_{j_{1}}\wedge\ldots\wedge v_{j_{i}}\wedge\ldots\wedge\widehat{w}_{k_{\kappa}}\wedge\ldots)w_{k_{\kappa}}.$$
This map is alternating in $v_{j_{1}},\ldots,v_{j_{i}}$ and $w_{k_{1}},\ldots,w_{k_{n+2-i}}$, hence it factors via the canonical ($H$-equivariant) map $V^{n+2}\to\bigwedge^{i}V\otimes\bigwedge^{n+2-i}V$ to a linear map $\bigwedge^{i}V\otimes\bigwedge^{n+2-i}V\to V$. 
This map is equal to $p_{i}$, which is seen on the basis elements, and $H$-equivariant. Hence $p_{i}$ is a linear $H$-equivariant map. Moreover, for $(J,K)\in\calJ_{i}\times\calJ_{n+2-i}$ with $J\cap K=\{r\}$ we have $p_{i}(e_{J}\otimes e_{K})=\pm e_{r}$, which shows that $p_{i}$ is surjective.
\end{proof}

\begin{corollary}\label{corollary: concrete embedding}
The embedding $V\to V^{G}_{(\omega_{i},\omega_{n+2-i})}$ is determined by $e_{1}\mapsto \sum_{J,K}\eps(J,K(1))e_{J}\otimes e_{K}$, where the sum is taken over the pairs $(J,K)\in\calJ_{i}\times\calJ_{n+2-i}$ such that $J\cap K=\{1\}$.
\end{corollary}

In order to write down the entries of this matrix we have to fix an ordering on the $M=T$-types that occur in $V$ which are given as $(k_{1},\ldots,k_{n+1})\in\bbN_{0}^{n+1}$ with $\sum_{i=1}^{n+1}k_{i}=1$. 
This corresponds to the standard basis $(e_{1},\ldots,e_{n+1})$ of $V=\bbC^{n+1}$. In this way $\End_{T}(V)\cong\bbC^{n+1}$.

The element $g_{a}\in\End(\bbC^{n+1})$ is determined by its first row, since the elements in the columns are all Weyl group translates, see Remark \ref{rk: Weyl group action}. The weight of a vector $e_{J}\otimes e_{K}$ is of the form $t\mapsto t_{j_{1}}\cdots t_{j_{i}}t_{k_{1}}\cdots t_{k_{n+2-1}}$, where $t_{1}\cdots t_{n+1}=1$.

\begin{theorem}\label{theorem: g_a}
The first row of $g_{a}$ is given as follows. The $m$-th element is the polynomial
$$(t_{1},\ldots,t_{n+1})\mapsto\binom{n}{m-1}^{-1} \sum_{(J,K)\in\calJ_{m}\times\calJ_{n+2-m}:J\cap K=\{1\}}\frac{t^{J}}{t^{K}},$$
where $t^{J}=t_{j_{1}}\cdots t_{j_{m}}$ and $t^{K}=t_{k_{1}}\cdots t_{k_{n+2-m}}$
\end{theorem}

\begin{proof}
Apply $(t,t^{-1})$ to the vector $\sum_{(J,K)\in\calJ_{m}\times\calJ_{n+2-m}:J\cap K=\{1\}}\eps(J,K(1))e_{J}\otimes e_{K}$ and then project down again by (\ref{eqn: projection}) to obtain the result.
\end{proof}

Let $\calJ^{(i)}_{m}$ denote the set of tuples $(j_{1},j_{2},\ldots,j_{m})\in\calJ_{m}$ such that $j_{p}=i$ for some $p=1,\ldots,m$. We can write
\begin{equation}\label{eq: rewritten coef Phi_0}
\binom{n}{m-1}^{-1} \sum_{(J,K)\in\calJ_{m}\times\calJ_{n+2-m}:J\cap K=\{i\}}\frac{t^{J}}{t^{K}}=\frac{t_{i}}{t_{1}\cdots t_{n+1}}\binom{n}{m-1}^{-1}\sum_{J\in\calJ_{m}^{(i)}}\left(t^{J\backslash\{i\}}\right)^{2}.
\end{equation}
We shall use this observation to calculate the polynomial factor $W_{\pol}^{\omega_{1}}$ of the weight matrix $W^{\omega_{1}}$. Recall from the discussion following Lemma \ref{lemma: Jacobian} that
$$W_{\pol}^{\omega_{1}}(\phi(t,t^{-1}))=\Phi_{0}^{\omega_{1}}(t,t^{-1})^{*}\Phi_{0}^{\omega_{1}}(t,t^{-1}),$$
which in this case amounts to the calculation of $g_{a}^{*}g_{a}$ in terms of the fundamental zonal spherical functions.

\begin{remark}\label{rk: flip for W}
Let $J:\bbC^{n+1}\to\bbC^{n+1}$ denote the linear mapping $e_{i}\mapsto e_{n+1-i}$. For $p\in\bbR[t^{\pm}_{1},\ldots,t^{\pm}_{n+1}]$ we have $\overline{p|_{A_{c}}(t)}=p|_{A_{c}}(t^{-1})$. This observation implies $\overline{\Phi^{\mu}_{0}(a)}=\Phi^{\mu}_{0}(a)J$. It follows that $W^{\omega_{1}}_{\pol}(\phi)=J\left(\Phi^{\omega_{1}}_{0}\right)^{t}\Phi_{0}^{\omega_{1}}$. Compare to the discussion following the proof of Lemma \ref{lemma: zonal SF}.
\end{remark}

\begin{theorem}
\label{thm:formula_weight}
The entries of $W_{\pol}^{\omega_{1}}$ are given by
\begin{multline}\nonumber
\binom{n}{j-1}\binom{n}{k-1}\left(W_{\pol}^{\omega_{1}}(\phi)\right)_{n+2-j,k}=\\
\sum_{r=0}^{\min(n+1-k,j-1)}(k+1-j+2r)\binom{n+1}{k+r}\binom{n+1}{j-1-r}\phi_{k+r}\phi_{j-1-r}
\end{multline}
where $\phi_{0}=\phi_{n+1}=1$ and where $j\le k$.
\end{theorem}

\begin{proof}
Let $\calJ_{j}^{(i)}=\{J\in\calJ_{j}\mid i\in J\}$. In view of Remark \ref{rk: flip for W} it is sufficient to show
\begin{equation}\label{eq: weight in phi}
\sum_{i=1}^{n+1}t_{i}^{2}\sum_{(J,K)\in\calJ_{j}^{(i)}\times\calJ_{k}^{(i)}}(t^{J\backslash\{i\}})^{2}(t^{K\backslash\{i\}})^{2}=
\sum_{r=0}^{\min(n+1-k,j-1)}(k+1-j+2r)\widetilde{\phi}_{k+r}\widetilde{\phi}_{j-1-r},
\end{equation}
where $\widetilde{\phi}_{m}=\binom{n+1}{m}\phi_{m}$ is the elementary symmetric function evaluated at $(t_1^2,\cdots, t_m^2)$.
This equality follows from the more general result in Proposition \ref{prop: equality of expression in symmetric functions} that we prove below. 
The specialization that yields \eqref{eq: weight in phi} is discussed below the proof of Proposition \ref{prop: equality of expression in symmetric functions}.
\end{proof}

To formulate Proposition \ref{prop: equality of expression in symmetric functions} we use the notation of \cite[\S 1.2]{Macd} 
$$e_r ( t_1,\cdots, t_{n+1}) =\sum_{1\leq j_1<j_2<\cdots j_r\leq n+1} t_{i_1}t_{i_2}\cdots t_{j_r},\quad 0 \leq r  \leq n+1$$
for the elementary symmetric functions, with the convention $e_0( t_1,\cdots, t_{n+1})=1$. The same notation is used in the proof of Lemma \ref{lem:barycenter}. 
For the proof of Proposition \ref{prop: equality of expression in symmetric functions} we do not need to assume that $e_{n+1}(t_1,\cdots , t_{n+1})=t_1\cdots t_{n+1}$ equals $1$. 
The generating function for the elementary symmetric functions is given by
\begin{equation}\label{eq:genfunelemtarysymmf}
\sum_{r=0}^{n+1} e_r ( t_1,\cdots, t_{n+1}) z^r = \prod_{i=1}^{n+1} (1+t_iz).
\end{equation}
To deal with the functions on the right hand side of \eqref{eq: rewritten coef Phi_0} we define
$$e_{p}^{(i)} ( t_1,\cdots,t_{n+1}) = \frac{\del}{\del t_{i}}e_{p+1}(t_{1},\ldots,t_{n+1}), \qquad 
\sum_{r=0}^{n} e^{(i)}_r ( t_1,\cdots, t_{n+1}) z^r = z \prod_{\stackrel{\scriptstyle{j=1}}{j\not=i}}^{n+1} (1+t_jz). $$
Applying $z\frac{d}{dz}$ (Euler operator) to \eqref{eq:genfunelemtarysymmf} and comparing the coefficients gives
\begin{equation}\label{eq:dergenfunelemtarysymmf}
 r e_r  = \sum_{i=1}^{n+1} t_i e_{r-1}^{(i)} 
\end{equation}
for $r\geq 1$ and for $r=0$ we interpret the right hand as zero by the convention that $e_{-k}=0$ for $k\in \mathbb{N}\setminus\{0\}$. 
We also follow the convention that $e_k=0$ for $k>n+1$.
For $0\leq r \leq N$ we have by \eqref{eq:dergenfunelemtarysymmf} 
\begin{gather}\label{eq:onetermconvterm}
(N-2r) e_{N-r} e_r = (N-r)e_{N-r}e_r - e_{N-r} re_r  =
\sum_{i=1}^{n+1} t_i \left( e_{N-r-1}^{(i)}e_r - e_{N-r} e_{r-1}^{(i)}\right) 
\end{gather}
which we want to rewrite as a telescoping sum. 

\begin{lemma}\label{lem:reducedifference}
$e^{(i)}_{m-1}e_k - e^{(i)}_{k-1} e_m = e^{(i)}_{m-1}e^{(i)}_k - e^{(i)}_{k-1} e^{(i)}_m$.
\end{lemma}

\begin{proof} We consider a generating function for the left hand side,
\begin{multline*}
\sum_{m=1}^{n+1} \sum_{k=1}^{n+1}\left( e^{(i)}_{m-1}e_k - e^{(i)}_{k-1} e_m\right) z^m w^k  =\\
z \prod_{\stackrel{\scriptstyle{l=1}}{l\not=i}}^{n+1} (1+t_lz) \left(\prod_{p=1}^{n+1} (1+t_pw) -1\right) 
- w \prod_{\stackrel{\scriptstyle{p=1}}{p\not=i}}^{n+1} (1+t_pw) \left(\prod_{l=1}^{n+1} (1+t_lz) -1\right) .
\end{multline*}
Working out the brackets, taking out the common factor in the double products, and simplifying gives 
products that are generating functions. This then equals 
\begin{gather*}
(z-w) \sum_{r=0}^{n} e^{(i)}_r z^r \sum_{s=0}^{n} e^{(i)}_s w^s - \sum_{r=0}^{n} e^{(i)}_r z^{r+1} +  \sum_{s=0}^{n} e^{(i)}_s w^{s+1} = \\
\sum_{r=1}^{n+1} \sum_{s=1}^{n} e^{(i)}_{r-1}  e^{(i)}_s z^{r} w^s 
- \sum_{r=1}^{n} \sum_{s=1}^{n+1} e^{(i)}_r e^{(i)}_{s-1} z^r  w^s
\end{gather*}
and comparing coefficients shows the result. 
\end{proof}

Applying Lemma \ref{lem:reducedifference} to \eqref{eq:onetermconvterm} proves the following. 

\begin{proposition}\label{prop: equality of expression in symmetric functions} 
For all $N,r\in\bbN_{0}$ with $r\le N$ the following identity holds,
\begin{gather*}
\sum_{r=a}^b (N-2r) e_{N-r} e_r = \sum_{i=1}^{n+1} t_i \left( e_{N-b-1}^{(i)}e^{(i)}_b - e^{(i)}_{N-a} e_{a-1}^{(i)}\right). 
\end{gather*}
\end{proposition}


Now pick $a=k$, $b = k + \min(n+1-k,j-1)$, $N=k+j-1$, and put $r=s+k$. Proposition \ref{prop: equality of expression in symmetric functions} yields
\begin{gather*}
\sum_{s=0}^{\min(n+1-k,j-1)} (j-k-1-2s) e_{j-1-r} e_{k+s} = 
- \sum_{i=1}^{n+1} t_i  e^{(i)}_{j-1} e_{k-1}^{(i)}
\end{gather*}
since $k+j-2-(k + \min(n+1-k,j-1))<0$ the corresponding term vanishes. This yields \eqref{eq: weight in phi} after
taking all arguments squared.

We now switch back to the group situation, and we assume that $t_1\cdots t_{n+1}=1$

\begin{proposition}
The function $a\mapsto\det(g_{a})$ is alternating in the sense that $\det(g_{w(a)})=\det(w)\det(g_{a})$ and 
$$
\det(g_{a})=c \prod_{i<j}(t^{2}_{i}-t^{2}_{j}), \qquad c=\prod_{m=0}^{n}\binom{n}{m-1}^{-1}.
$$
\end{proposition}

\begin{proof}
From (\ref{eq: rewritten coef Phi_0}) and $t_1\cdots t_{n+1}=1$ we see that the entries $g_{a}$ are regular functions in the variables 
$t^{2}_{1},\ldots,t^{2}_{n+1}$. The degree of this function in these variables is equal to the number of reflections in $S_{n+1}$ and 
this function is alternating by definition. Following \cite[Prop.3.13(b)]{Humphreys} we conclude that it is a multiple of the Jacobian of the basic invariants. 
The multiple is calculated using Theorem \ref{theorem: g_a}.
\end{proof}


\subsection{Irreducibility of the weight}

Now we study the irreducibility of the weight $W^{\omega_{1}}_{\pol}$. We say that the matrix weight $W$, i.e. a function defined on a set $S$ taking values in 
the self-adjoint matrices of size $N\times N$, reduces to weights of smaller size if there exists a constant matrix $M$ and weights $W_1,\ldots,W_k$ of lower size  
such that $MW(x)M^\ast$ is equal to the block diagonal matrix $\diag(W_1(x),\ldots, W_k(x))$ for all $x\in S$. 
In such a case, the real vector space
$$
{\mathcal{A}}_{W}= \{ Y \in\End(\bbC^{N}) \mid YW(x) = W(x)Y^\ast, \quad \text{for all }x\in S\},
$$
is non-trivial.  If the subspace $A_h$ of self-adjoint elements in the commutant algebra
$$
A_W= \{Y \in \End(\bbC^{N}) \mid YW(x) = W(x)Y,\quad \text{for all }x\in S\},
$$
is nontrivial, then $W$ is reducible via a unitary matrix $M$. In \cite{KR15} we prove that 
$\mathcal{A}_{W}$ is $\ast$-invariant if and only if  $\mathcal{A}_{W}=(A_{W})_h$. 
We will show that, for $N=n+1$ with $n>1$, and $S=\phi(A_c)$,  the weight $W_{\mathrm{pol}}^{\omega_{1}}$ is irreducible by showing that 
$A_{W^{\omega_{1}}_{\mathrm{pol}}}$ is trivial and that $\mathcal{A}_{W^{\omega_{1}}_{\mathrm{pol}}}$ is $\ast$-invariant.

\begin{theorem}
\label{thm:commutator}
For $n\geq 2$, the commutant algebra $A_{W^{\omega_{1}}_{\mathrm{pol}}}$ is trivial, i.e. it consists of multiples of the identity matrix. Moreover, the real vector space $\mathcal{A}_{W^{\omega_{1}}_{\mathrm{pol}}}$ is $\ast$-invariant.
\end{theorem}

\begin{proof}
For $i=1,\ldots,\lfloor \tfrac{n+1}{2}\rfloor$, we denote by $W_{(i)}$ the coefficient of $\phi_i\phi_{n+1-i}$ in $W^{\omega_{1}}_{\mathrm{pol}}$. It follows from Theorem \ref{thm:formula_weight} that $W_{(i)}$ is given by
\begin{equation}
\label{eq:weight_W(i)}
W_{(i)}=\sum_{k=i}^{n-i} (n+1-2i) {\binom{n+1}{i}\binom{n+1}{n+1-i}} {{\binom{n}{n+1-k}}^{-1} {\binom{n}{k-1}}^{-1} }\, E_{k,k},
\end{equation}
where $E_{k,j}$ denotes the matrix with a one in the $(k,j)$-th entry and zero elsewhere. Note that the first and last $i$ diagonal entries of $W_{(i)}$ are zero.

First we prove that the commutant algebra is trivial. Let $Y\in A_{W^{\omega_{1}}_{\mathrm{pol}}}$. Since the spherical functions $\phi_1,\ldots,\phi_n$ are algebraically independent, it follows from Theorem \ref{thm:formula_weight} that $YW_{(i)}-W_{(i)}Y=0$ for all $i=1,\ldots,\lfloor \tfrac{n+1}{2}\rfloor$. If we set $i=1$ in this equation, since $(W_{(1)})_{11}=(W_{(1)})_{n+1,n+1}=0$, the first and last rows and columns give
$$Y_{1j}(W_{(1)})_{jj}=0,\quad Y_{nj}(W_{(1)})_{jj}=0,\quad Y_{j1}(W_{(1)})_{jj}=0,\quad Y_{jn}(W_{(1)})_{jj}=0, \qquad  j=2,\dots,n,$$
which implies $Y_{1j}=Y_{nj}= Y_{j1} = Y_{jn}=0$ for $j=2,\ldots n$ by \eqref{eq:weight_W(i)}.
Repeating this process for $i=2,\ldots,\lfloor \tfrac{n+1}{2}\rfloor$ we obtain that the only possible non-zero entries of $Y$ are of the form $Y_{kk}$ and $Y_{k,n+2-k}$ for $k=1,\ldots,n+1$. The coefficient of $\phi_1$ in $W_{\mathrm{pol}}$ is the matrix
$$W_{(\phi_1)}=\sum_{k=1}^n {n(n+1)}{\binom{n}{n+1-k}^{-1} \binom{n}{k}^{-1}} \, E_{k,k+1} + (n+1)E_{n+1,1}.$$
The $(k,k+1)$-th entry of $YW_{(\phi_1)}-W_{(\phi_1)}Y=0$ gives
$$(Y_{kk}-Y_{k+1,k+1}) {n(n+1)}{\binom{n}{n+1-k}^{-1}\binom{n}{k}^{-1}}=0,$$
which implies that $Y_{kk}=Y_{k+1,k+1}$ for $k=1,\ldots,n$. The $(n+1-k,k)$-th entries of  $YW_{(\phi_1)}-W_{(\phi_1)}Y=0$ give that $Y$ is a multiple of the identity. This proves that the commutant algebra of $W$ is trivial.

Now we prove the $\ast$-invariance of $\mathcal{A}_{W^{\omega_{1}}_{\mathrm{pol}}}$. For $Y\in \mathcal{A}_{W^{\omega_{1}}_{\mathrm{pol}}}$, we will show that $Y^\ast\in \mathcal{A}_{W^{\omega_{1}}_{\mathrm{pol}}}$. The $(k,j)$-th entry of the equation $YW_{(0)}=W_{(0)}Y^\ast$ gives
\begin{equation}
\label{eq:entry_commA}
Y_{k,j}= {\binom{n}{n+1-k}\binom{n}{k-1}} {\binom{n}{n+1-j}^{-1}\binom{n}{j-1}^{-1}} \overline Y_{j,k}.
\end{equation}
It is immediate from \eqref{eq:entry_commA} that the diagonal elements $Y_{k,k}$ are real and that $Y_{k,n+2-k}=\overline Y_{n+2-k,k}$, for $k=1,\ldots,n+1$. Now it its enough to prove that $Y_{k,j}=0$ if $k\neq j$ or $k\neq n+2-j$. For this we proceed as for the commutant algebra. Since $(W_{(1)})_{11}=(W_{(1)})_{n+1,n+1}=0$, the first and last rows of the equation $YW_{(1)}=W_{(1)}Y^\ast$ give
$$Y_{1j}(W_{(1)})_{jj}=0,\qquad Y_{nj}(W_{(1)})_{jj}=0,\qquad j=2,\ldots n.$$
This implies $Y_{1j}=Y_{nj}=0$ for $j=2,\ldots n$, since $(W_{(1)})_{kk}\neq 0$. The first row and column of the equation $YW_{(0)}=W_{(0)}Y^\ast$ implies now that $Y_{kn}=Y_{j1}=0$ for $j=2,\ldots n$. If we proceed in the same way for the equation $YW_{(i)}=W_{(i)}Y^\ast$ with $i>1$ we obtain that $Y_{kj}=0$ unless $k=j$ or $k=n+2-j$. This completes the proof of the theorem.
\end{proof}

\begin{corollary}
The matrix weight $W^{\omega_{1}}_{\mathrm{pol}}$ is indecomposable.
\end{corollary}
\begin{proof}
Since the real vector space $\mathcal{A}_{W^{\omega_{1}}_{\mathrm{pol}}}$ is $\ast$-invariant, it follows from \cite[Corollary 2.5]{KR15} that $\mathcal{A}_{W^{\omega_{1}}_{\mathrm{pol}}}$ is the set of self-adjoint elements in the commutant algebra $A_{W^{\omega_{1}}_{\mathrm{pol}}}$. Since the commutant algebra is trivial by Theorem \ref{thm:commutator}, $\mathcal{A}_{W^{\omega_{1}}_{\mathrm{pol}}}$ consists on the real multiples of the identity matrix. Thus $W^{\omega_{1}}_{\mathrm{pol}}$ is indecomposable.
\end{proof}

\section{Differential properties}

Let $G,H$ be as above and $\mu=k\omega_{1}$. Then the center $Z(\lag)\cong Z(\lasl(n+1,\bbC))\otimes Z(\lasl(n+1,\bbC))$ of $U(\lag)$ contains the two Casimir operators $\Omega_{L}=\Omega\otimes1$ and $\Omega_{R}=1\otimes\Omega$, where $\Omega\in Z(\lasl(n+1,\bbC))$ is the Casimir operator of order two. Let $D_{L}^{\mu},D_{R}^{\mu}\in\calD_{\mu}$ denote their images in $\calD(\mu)$ under the map $\calD^{\mu}$, see Subsection \ref{subsection:diffoperatorsMVOP}.

The operators $\Omega_{L}$ and $\Omega_{R}$ act on $V^{G}_{(\lambda_{1},\lambda_{2})}$ by multiplication with the scalars $\gamma(\Omega_{L},\lambda)=|\lambda_{1}+\rho|^{2}-|\rho|^{2}$ and $\gamma(\Omega_{R},\lambda)=|\lambda_{2}+\rho|^{2}-|\rho|^{2}$ respectively, where $\gamma$ is the Harish-Chandra isomorphism. Note that $\gamma_{\mu}=\gamma$ on the image of $Z(\lag)\to\bbD(\mu)$.

We denote the diagonal eigenvalue matrices of $\Omega_{L}$ and $\Omega_{R}$ on the eigenfunction $\Phi^{\mu}_{d}$ by $\Gamma^{\mu}_{L,d}$ and $\Gamma^{\mu}_{R,d}$ respectively. 

The radial part operator $\rad_{\mu}$ respects the degree of differentiation and so does conjugation with $\Phi^{\mu}_{0}$ and changing the variables. Hence the images of $\Omega_{L}$ and $\Omega_{R}$ under $\calD$ are differential operators of order two with matrix-valued polynomials as coefficients. We denote these images by $D^{\mu}_{L}$ and $D^{\mu}_{R}$ respectively.

\begin{lemma}
The differential operator $D^{\mu}_{L}-D^{\mu}_{R}$ has order $\le1$. It has the polynomials $Q^{\mu}_{d}$ as simultaneous eigenfunctions with eigenvalues $\Gamma^{\mu}_{L,d}-\Gamma^{\mu}_{R,d}$.
\end{lemma}

\begin{proof}
Let $(H_{1},\ldots,H_{n})$ be an orthonormal basis of $\lat$ with respect to the Killing form. We have
$$\Omega=\sum_{i=1}^{n}H_{i}^{2}+\sum_{\alpha\in\Delta(\SL(n+1,\bbC),T)}E_{\alpha}E_{-\alpha},$$
where $E_{\alpha}$ is a root vector with $(E_{\alpha},E_{-\alpha})=1$. The Killing form on $\lat\oplus\lat$ is given by the sum of the Killing forms on the summands. Hence
$$\left( (H_{i},-H_{i})/\sqrt{2},i=1,\ldots,n\right) \cup \left( (H_{i},H_{i})/\sqrt{2},i=1,\ldots,n\right)$$
is an orthonormal basis of $\lat\oplus\lat=\laa\oplus\lat_{M}$. We have
\begin{eqnarray*}
\sum_{i=1}^{n}((H_{i},0)^{2}+(0,H_{i})^{2})&=&\sum_{i=1}^{n}((H_{i},-H_{i})^{2}+(H_{i},H_{i})^{2}),\\
\sum_{i=1}^{n}((H_{i},0)^{2}-(0,H_{i})^{2})&=&2\sum_{i=1}^{n}(H_{i},H_{i})(H_{i},-H_{i}).
\end{eqnarray*}
This shows that
$$\Omega_{L}-\Omega_{R}=2\sum_{i=1}^{n}(H_{i},H_{i})(H_{i},-H_{i})+\mbox{other terms},$$
and hence that $D^{\mu}_{L}-D^{\mu}_{R}$ has order one if $\pi^{H}_{\mu}|_{\lam}$ is not trivial and order zero otherwise. This proves the statement.
\end{proof}

To be able to calculate this order one differential operator explicitly we continue our analysis. Write $\xi_{i}=(H_{i},-H_{i})/\sqrt{2}$. The $\mu$-radial part is of the form
$$\rad_{\mu}(\Omega_{L}-\Omega_{R})=\sum_{i=1}^{n}\pi^{H}_{\mu}(H_{i})\del_{\xi_{i}}+G^{\mu},$$
where $G^{\mu}$ is an $\End(\End_{M_{c}}(V^{H}_{k\omega_{1}}))$-valued function on $A$. Conjugating with $\Phi^{\mu}_{0}$ yields
\begin{multline}
\label{eq:First_OrderDO}
(D^{\mu}_{L}-D^{\mu}_{R})Q(\phi)\\
=\sum_{k=1}^{n}\left(\sum_{i=1}^{n}m_{(\Psi^{\mu}_{0})^{-1}}\pi^{H}_{\mu}(H_{i})m_{(\Psi^{\mu}_{0})}\del_{\xi_{i}}\phi_{k}\right)(\del_{k}Q)(\phi)+(\Gamma^{\mu}_{L,0}-\Gamma^{\mu}_{R,0})Q(\phi).
\end{multline}
As a consequence of Proposition \ref{prop: do total degrees} we see that the expression
\begin{equation}
\label{eq:upsilon_k}
\Upsilon^{\mu}_{\ell}(\phi)=\sum_{i=1}^{n}m_{(\Psi^{\mu}_{0})^{-1}}\pi^{H}_{\mu}(H_{i})m_{(\Psi^{\mu}_{0})}\del_{\xi_{i}}\phi_{\ell},
\end{equation}
is matrix-valued polynomial of degree one.


\section{Examples}
In this section we give explicit expressions for the orthogonality weights and differential operators developed in the previous sections for small $n$ and for $k=1$. The polynomial part for the weight matrix is given for any $n$ in Theorem \ref{thm:formula_weight} and the scalar part of the weight is given in Theorem \ref{thm:scalar_weight}.
For this section we have complemented the theory of the previous sections by calculations using computer algebra. 

In order to compute the radial part of the Casimir operator $D^{\mu}_{L}+D^{\mu}_{R}$, we use the first order differential equations in Lemma \ref{lem:firstorderDEsPhi0}. For the first order differential operator $D^{\mu}_{L}-D^{\mu}_{R}$ we use \eqref{eq:First_OrderDO}. Using the explicit expression for $\Psi_0$ given in Theorem \ref{theorem: g_a}, we compute explicitly its inverse and after some simplification we obtain the matrices $L_k(\phi)$ and $C_k$ in \eqref{eq:CasimironQx} and the matrices $\Upsilon_\ell$ in \eqref{eq:upsilon_k}.

\subsection{The case $n=2$, $k=1$}
This case is the simplest nontrivial example of matrix-valued orthogonal polynomials in two variables. We drop the weight $\mu=\omega_{1}$ in the notation of what follows.

\subsubsection{The orthogonality}
By Theorem \ref{thm:expression_Phi0}, the function $\Psi_0$ is given explicitly by
\begin{eqnarray}
\label{eq:Phi_0_n=2}
\Psi_{0}(t,t^{-1})=\begin{pmatrix}
t_{1} & \frac{1}{2}(t_{3}^{-1}t_{2}+t_{2}^{-1}t_{3}) & t_{1}^{-1} \\
t_{2} & \frac{1}{2}(t_{1}^{-1}t_{3}+t_{3}^{-1}t_{1}) & t_{2}^{-1} \\
t_{3} & \frac{1}{2}(t_{2}^{-1}t_{1}+t_{1}^{-1}t_{2}) & t_{3}^{-1}
\end{pmatrix}, \quad t_{1}t_{2}t_{3}=1.
\end{eqnarray}
The zonal spherical functions are
$$\phi_1(t,t^{-1})=\frac{1}{3}\left(t_1^2+t_2^2+t_3^2\right),\qquad \phi_2(t,t^{-1})=\frac{1}{3}\left(t_1^2t_2^2+t_1^2t_3^2+t_2^2t_3^2\right).$$
The matrix-valued orthogonality relations for the polynomials $Q_d^{\omega_1}$ of degree $d\in \mathbb{N}_0\times \mathbb{N}_0$ follow directly from \eqref{eq:orthoCNonA} and Theorem \ref{thm:formula_weight}. We have
$$\int_{\phi(\exp(\mathfrak{b}))} (Q_d^{\omega_1}(\phi))^\ast \, W(\phi)\, Q_{d'}^{\omega_1}(\phi) \, d\phi=\delta_{d,d'}\, H_{d},$$
where $H_d$ is a constant matrix and the matrix weight $W(\phi)=w(\phi)W_{\mathrm{pol}}(\phi)$ is given by
\begin{multline}\label{eq: weight (2,1)}
W_{\mathrm{pol}}(\phi)=\begin{pmatrix}
3&3\phi_{1}&3\phi_{2}\\
3\phi_{2} & (9\phi_{1}\phi_{2}+3)/4&3\phi_{1}\\
3\phi_{1} & 3\phi_{2} & 3
\end{pmatrix},\\
w(\phi) = \frac{9}{4\pi^{2}}(-\phi_1^2\phi_2^2+4\phi_1^3+4\phi_2^3-18\phi_1\phi_2+27)^{\frac12}.
\end{multline}

\subsubsection{The differential operators}
We take the orthogonal basis of $\mathfrak{t}$ with respect to the Killing form $(H_1,H_2)$, where $H_1=\frac{\sqrt{2}}{2}\diag(1,-1,0)$, 
$H_2=\frac{\sqrt{6}}{6}\diag(1,1,-2)$. The derivatives $\partial_{\xi_i}$ are given by
$$\partial_{\xi_1}=\frac{\sqrt{2}}{2}\left( t_1\partial_{t_1}-t_2\partial_{t_2}\right),\quad 
\partial_{\xi_2}=\frac{\sqrt{6}}{{6}}\left(t_1\,\partial_{t_1}+t_2\,\partial_{t_2}-2\,t_3\,\partial_{t_3}\right).$$
The explicit expression of the radial part of the Casimir operator follows from \eqref{eq:CasimironQx} and the explicit expression of $\Psi_0$ given in \eqref{eq:Phi_0_n=2}. 
Explicitly we have
\begin{align*}
(\del_{\xi_{1}}\phi_{1})(\del_{\xi_{1}}\phi_{1})+(\del_{\xi_{2}}\phi_{1})(\del_{\xi_{2}}\phi_{1})&=\frac{8}{3}(\phi_1^2-\phi_2),\\
(\del_{\xi_{1}}\phi_{1})(\del_{\xi_{1}}\phi_{2})+(\del_{\xi_{2}}\phi_{1})(\del_{\xi_{2}}\phi_{2})&=\frac{4}{3}(\phi_1\phi_2-1)=(\del_{\xi_{1}}\phi_{2})(\del_{\xi_{1}}\phi_{1})+(\del_{\xi_{2}}\phi_{2})(\del_{\xi_{2}}\phi_{1})\\
(\del_{\xi_{1}}\phi_{2})(\del_{\xi_{1}}\phi_{2})+(\del_{\xi_{2}}\phi_{2})(\del_{\xi_{2}}\phi_{2})&=\frac{8}{3}(\phi_2^2-\phi_1)
\end{align*}
A straightforward computation shows that
\begin{align*}
L_1(\phi_1,\phi_2)=\begin{pmatrix} \frac{8}{3}\phi_1 & -2\phi_2 & 0 \\ 0  & 4 \phi_1 & 0 \\ 0& 0 & \frac{4}{3}\phi_1 \end{pmatrix},\qquad C_1=\begin{pmatrix} 0 & 0 & -\frac{4}{3} \\ -\frac{8}{3}  & 0 & 0 \\ 0& -2 & 0 \end{pmatrix}, \\
L_2(\phi_1,\phi_2)=\begin{pmatrix} \frac{4}{3}\phi_2 & 0 & 0 \\ 0  & 4 \phi_2 & 0 \\ 0& -2\phi_1 & \frac{8}{3}\phi_2 \end{pmatrix},\qquad C_2=\begin{pmatrix} 0 & -2 &0 \\ 0  & 0 & -\frac{8}{3} \\  -\frac{4}{3}&0 & 0 \end{pmatrix}.
\end{align*}
The coefficient of order zero $\Gamma_0$ is given by
$$
\Gamma_0=\Gamma_{L,0}+\Gamma_{R,0} = \mathrm{diag}(\frac83, \frac{16}{3}, \frac83).
$$
We recall that $\Gamma_0$ is also the eigenvalue of the polynomial $Q_{0,0}$. Moreover, the eigenvalue of the polynomial $Q_{d_1,d_2}$ is given by
the diagonal matrix 
$$\Gamma^{+}_{d_1,d_2}=\Gamma_{L,d_1,d_2}+\Gamma_{R,d_1,d_2}= (\frac43 d_1^2 + \frac43 d_1d_2 + \frac{4}{3} d_2^2)\mathrm{I}+\mathrm{diag} \begin{pmatrix}
\frac{16}{3} d_1 + \frac{14}{3} d_2 +\frac{8}{3} \\
 6 d_1 + 6 d_2 +\frac{16}{3} \\
\frac{14}{3} d_1 + \frac{16}{3} d_2 +\frac{8}{3}
\end{pmatrix}.$$
The first order differential operator \eqref{eq:First_OrderDO} is obtained directly from the expression of $\Psi_0$ . We get
$$D_L-D_R=\Upsilon_{1}(\phi)\, \partial_1 + \Upsilon_{2}(\phi) \, \partial_2 + (\Gamma_{L,0}-\Gamma_{R,0}),$$
where
\begin{align*}
\Upsilon_{1}(\phi)&=(\Psi_{0})^{-1}\pi^{H}_{\omega_1}(H_{1})\Psi_{0} \, \del_{\xi_{1}}\phi_{1}+(\Psi_{0})^{-1}\pi^{H}_{\omega_1}(H_{2})\Psi_{0}\, \del_{\xi_{2}}\phi_{1} = \begin{pmatrix} \frac43 \phi_1 & \phi_2 & \frac23 \\ -\frac43 & -\frac23\phi_1 & 0 \\ 0 & -\frac13 & -\frac23 \phi_1 \end{pmatrix},\\
\Upsilon_{2}(\phi)&=(\Psi_{0})^{-1}\pi^{H}_{\omega_1}(H_{1})\Psi_{0} \, \del_{\xi_{1}}\phi_{2}+(\Psi_{0})^{-1}\pi^{H}_{\omega_1}(H_{2})\Psi_{0}\, \del_{\xi_{2}}\phi_{2} = \begin{pmatrix} \frac23 \phi_2 & \frac13 & 0 \\ 0 & \frac23\phi_2 & \frac43 \\ -\frac23 &  -\phi_1 & -\frac43 \phi_2 \end{pmatrix},
\end{align*}
The coefficient of order zero for the first order differential operator is given by
$$\Gamma_{L,0}-\Gamma_{R,0} = \mathrm{diag}(\frac83,0,-\frac83).
$$
Moreover, the eigenvalue of the polynomial $Q_{d_1,d_2}$ is given by the diagonal matrix 
$$\Gamma^{-}_{d_1,d_2}=\Gamma_{L,d_1,d_2}-\Gamma_{R,d_1,d_2}= \mathrm{diag} \begin{pmatrix}
-\frac{4}{3} d_1 + \frac{2}{3} d_2 +\frac{8}{3} \\
 -\frac{2}{3} d_1 + \frac{2}{3} d_2  \\
- \frac{2}{3} d_1 - \frac{4}{3} d_2 -\frac{8}{3}
\end{pmatrix}.$$

\subsection{The case $n=3$, $k=1$}
Here we obtain a $4\times 4$ matrix weight in three variables. We drop the weight $\mu=\omega_{1}$ in the notation of what follows.

\subsubsection{The orthogonality}
The function $\Psi_0$ is given by
\begin{eqnarray*}
\Psi_{0}(t,t^{-1})=\begin{pmatrix}
t_1 & \frac{1}{3}\left( \frac{t_2}{t_3t_4}+\frac{t_3}{t_2t_4}+\frac{t_4}{t_2t_3}\right) &  \frac{1}{3}\left( \frac{t_2t_3}{t_4}+\frac{t_2t_4}{t_3}+\frac{t_3t_4}{t_2}\right)  & t_1^{-1} \\
t_2 &\frac{1}{3}\left( \frac{t_1}{t_3t_4}+\frac{t_3}{t_1t_4}+\frac{t_4}{t_1t_3}\right) & \frac{1}{3}\left( \frac{t_1t_3}{t_4}+\frac{t_1t_4}{t_3}+\frac{t_3t_4}{t_1}\right)  & t_2^{-1} \\
t_3 & \frac{1}{3}\left( \frac{t_1}{t_2t_4}+\frac{t_2}{t_1t_4}+\frac{t_4}{t_1t_2}\right) &  \frac{1}{3}\left( \frac{t_1t_2}{t_4}+\frac{t_1t_4}{t_2}+\frac{t_2t_4}{t_1}\right)  & t_3^{-1} \\
t_4 & \frac{1}{3}\left( \frac{t_1}{t_2t_3}+\frac{t_2}{t_1t_3}+\frac{t_3}{t_1t_2}\right) &  \frac{1}{3}\left( \frac{t_1t_2}{t_3}+\frac{t_1t_3}{t_2}+\frac{t_2t_3}{t_1}\right)  & t_4^{-1}
\end{pmatrix}, \quad t_{1}t_{2}t_{3}t_{4}=1.
\end{eqnarray*}
The zonal spherical functions are
\begin{gather*}
\phi_1(t,t^{-1})=\frac{1}{4}\left(t_1^2+t_2^2+t_3^2 +t_4^2\right),\quad \phi_2(t,t^{-1})=\frac{1}{6}\left(t_1^2t_2^2+t_1^2t_3^2+t_1^2t_4^2+t_2^2t_3^2+t_2^2t_4^2+t_3^2t_4^2\right), \\
\phi_3(t,t^{-1})=\frac{1}{4}\left(t_1^2t_2^2t_3^2+t_1^2t_2^2t_4^2+t_1^2t_3^2t_4^2+t_2^2t_3^2t_4^2\right).
\end{gather*}
The matrix weight $W(\phi)=w(\phi)W_{\mathrm{pol}}(\phi)$ is given by
\begin{gather*}
W_{\mathrm{pol}}(\phi)=\begin{pmatrix}
4 & 4\phi_1 & 4\phi_2 & 4\phi_3 \\ 
4\phi_3  & \frac{32}{9} \phi_1\phi_3 +\frac{1}{3} & \frac{8}{3}\phi_2\phi_3-\frac{4}{3}\phi_1 & 4\phi_2 \\
4\phi_2  & \frac{8}{3}\phi_1\phi_2 +\frac{4}{3}\phi_3 & \frac{32}{9} \phi_1\phi_3+\frac{1}{3} & 4\phi_1 \\
4\phi_1  & 4\phi_2 & 4\phi_3 & 4
\end{pmatrix}, 
\end{gather*}
\begin{multline}\label{eq: weight (3,1)}
w(\phi)= \frac{12}{\pi^3}\left(13824\phi_1^2\phi_2-3072\phi_1\phi_3-16384\phi_1^3\phi_3^3-13824\phi_1^2\phi_2^3-1536\phi_1^2\phi_3^2\right.\\
\left. -13824\phi_2^3\phi_3^2+13824\phi_2\phi_3^2-6912\phi_1^4-4608\phi_2^2+9216\phi_1^2\phi_2^2\phi_3^2+27648\phi_1^3\phi_2\phi_3\right.\\
\left.+27648\phi_2\phi_1\phi_3^3-46080\phi_1\phi_2^2\phi_3+20736\phi_2^4-6912\phi_3^4+256\right)^{\frac12}.
\end{multline}

\subsubsection{The differential operators}
We take the orthogonal basis of $\mathfrak{t}$ with respect to the Killing form $(H_1,H_2,H_3)$, where $H_1=\frac{\sqrt{2}}{2}\diag(1,-1,0,0)$, 
$H_2=\frac{\sqrt{6}}{6}\diag(1,1,-2,0)$, $H_3=\frac{\sqrt{3}}{6}(1,1,1,-3)$. The derivatives $\partial_{\xi_i}$ are given by
\begin{gather*}
\partial_{\xi_1}=\frac{\sqrt{2}}{2}\left( t_1\, \partial_{t_1}-t_2\,\partial_{t_2}\right),\quad 
\partial_{\xi_2}=\frac{\sqrt{6}}{{6}}\left(t_1\,\partial_{t_1}+t_2\,\partial_{t_2}-2\,t_3\,\partial_{t_3}\right),\\
\partial_{\xi_3}=\frac{\sqrt{3}}{{6}}\left(t_1\,\partial_{t_1}+t_2\,\partial_{t_2}+t_3\,\partial_{t_3}-3\,t_4\,\partial_{t_4}\right),
\end{gather*}
The explicit expression of the radial part of the Casimir operator follows from \eqref{eq:CasimironQx} and the explicit expression of $\Psi_0$ given in \eqref{eq:Phi_0_n=2}. Explicitly we have
\begin{align*}
(\del_{\xi_{1}}\phi_{1})(\del_{\xi_{1}}\phi_{1})+(\del_{\xi_{2}}\phi_{1})(\del_{\xi_{2}}\phi_{1})+(\del_{\xi_{3}}\phi_{1})(\del_{\xi_{3}}\phi_{1})&=3(\phi_1^2-\phi_2),\\
(\del_{\xi_{1}}\phi_{1})(\del_{\xi_{1}}\phi_{2})+(\del_{\xi_{2}}\phi_{1})(\del_{\xi_{2}}\phi_{2})+(\del_{\xi_{3}}\phi_{1})(\del_{\xi_{3}}\phi_{2})&=2(\phi_1\phi_2-\phi_3),\\
(\del_{\xi_{1}}\phi_{1})(\del_{\xi_{1}}\phi_{3})+(\del_{\xi_{2}}\phi_{1})(\del_{\xi_{2}}\phi_{3})+(\del_{\xi_{3}}\phi_{1})(\del_{\xi_{3}}\phi_{3})&=\phi_1\phi_3-1,\displaybreak[0] \\
(\del_{\xi_{1}}\phi_{2})(\del_{\xi_{1}}\phi_{2})+(\del_{\xi_{2}}\phi_{2})(\del_{\xi_{2}}\phi_{2})+(\del_{\xi_{3}}\phi_{2})(\del_{\xi_{3}}\phi_{2})&=\frac49\phi_2^2-\frac{32}{9}\phi_1\phi_3-\frac49, \displaybreak[0] \\
(\del_{\xi_{1}}\phi_{2})(\del_{\xi_{1}}\phi_{3})+(\del_{\xi_{2}}\phi_{2})(\del_{\xi_{2}}\phi_{3})+(\del_{\xi_{3}}\phi_{2})(\del_{\xi_{3}}\phi_{3})&=2(\phi_2\phi_3-\phi_1), \displaybreak[0] \\
(\del_{\xi_{1}}\phi_{3})(\del_{\xi_{1}}\phi_{3})+(\del_{\xi_{2}}\phi_{3})(\del_{\xi_{2}}\phi_{3})+(\del_{\xi_{3}}\phi_{3})(\del_{\xi_{3}}\phi_{3})&=3(\phi_3^2-\phi_2).
\end{align*}
A straightforward computation shows that
\begin{align*}
L_1(\phi_1,\phi_2)=\begin{pmatrix} 3\phi_1 & -2\phi_2 & -\frac{4}{3}\phi_3& 0 \\ 0  & 5 \phi_1 & 0 & 0 \\ 0& 0 & 3\phi_1 & 0 \\ 0 & 0 & 0 & \phi_1 \end{pmatrix},\qquad C_1=\begin{pmatrix} 0 & 0 & 0 & -1 \\ -3  & 0 & 0 & 0 \\ 0& -3 & 0 & 0 \\ 0 & 0 & -\frac53 & 0 \end{pmatrix}, \displaybreak[0] \\
L_2(\phi_1,\phi_2)=\begin{pmatrix} 2\phi_2 & -\frac83 \phi_3 & 0& 0 \\ 0  & 4 \phi_2 & -\frac83\phi_3 & 0 \\ 0& -\frac83 \phi_1 & 4\phi_2 & 0 \\ 0 & 0 & -\frac83\phi_1 & \frac43\phi_2 \end{pmatrix},\qquad C_2=\begin{pmatrix} 0 & 0 & -\frac23 & 0 \\ 0  & 0 & 0 & -2 \\ -2& 0 & 0 & 0 \\ 0 & -\frac23 & 0 & 0 \end{pmatrix}, \displaybreak[0] \\
L_3(\phi_1,\phi_2)=\begin{pmatrix} \phi_3 & 0 & 0 & 0 \\ 0  & 3 \phi_3 & 0 & 0 \\ 0& 0 & 5\phi_3 & 0 \\ 0 & -\frac43\phi1 & -2\phi_2 & 3\phi_3 \end{pmatrix},\qquad C_3=\begin{pmatrix} 0 & \frac53 & 0 & 0 \\ 0  & 0 & -3 & 0 \\ 0& 0 & 0 & -3 \\ -1 & 0 & 0 & 0 \end{pmatrix}.
\end{align*}
The coefficient of order zero is given by
$$\Gamma_0= \mathrm{diag}(\frac{15}{4},\frac{35}{4},\frac{35}{4},\frac{15}{4})
$$
Moreover, the eigenvalue of the polynomial $Q_{d_1,d_2}$ is given by
$$\Gamma^{+}_{d_1,d_2}= (\frac {3}{2} d_1^2 + 2 d_2^2 + \frac{3}{2} d_3^2+ 2 d_1d_2 +d_1d_3+2d_2d_3)\mathrm{I}+\mathrm{diag} \begin{pmatrix}
  \frac{15}{2} d_1 + 9d_2 + \frac{13}{2}d_3+\frac{15}{4} \\
  \frac{17}{2}d_1 + 11d_2 + \frac{15}{2}d_3+\frac{35}{4}\\
  \frac{15}{2} d_1 + 11d_2 + \frac{17}{2}d_3+\frac{35}{4}\\
  \frac{13}{2} d_1 + 9d_2 + \frac{15}{2}d_3+\frac{15}{4}
\end{pmatrix}.$$
The first order differential operator \eqref{eq:First_OrderDO} is obtained directly from the expression of $\Psi_0$ . We get
$$D_L-D_R=\Upsilon_{1}(\phi)\, \partial_1 + \Upsilon_{2}(\phi) \, \partial_2 + (\Gamma_{L,0}-\Gamma_{R,0}),$$
where
\begin{align*}
\Upsilon_{1}(\phi)& = \begin{pmatrix}
\frac34 \phi_1 & \phi_2 & \frac23 \phi_3 & \frac12 \\
-\frac32 & -\frac14\phi_1 & 0 & 0 \\
0 & -\frac12 & -\frac14 \phi_1 & 0 \\
0 & 0 & -\frac16 & -\frac14\phi_1
\end{pmatrix},  \qquad &\Upsilon_{2}(\phi)=&\begin{pmatrix} 
\phi_2 & \frac49 \phi_3 & \frac19 & 0 \\
0 & \phi_2 & \frac43 \phi_3 & 1 \\
-1 & -\frac43 \phi_1 & -\phi_2 & 0 \\
0 & -\frac19 & -\frac49 \phi_1 & -\phi_2
 \end{pmatrix}, \displaybreak[0] \\
\Upsilon_{3}(\phi)&=\begin{pmatrix} 
\frac14\phi_3 & \frac16 \phi_3 & 0 & 0 \\
0 &\frac14 \phi_3 & \frac12  & 0 \\
0 &  \phi_1 & \frac14 \phi_3 & \frac32 \\
-\frac12 & -\frac46\phi_1 & -\phi_2  & -\frac34\phi_3
 \end{pmatrix},\qquad &\Gamma_{L,0}-\Gamma_{R,0} =&
\begin{pmatrix}
\frac{15}{4} & 0 & 0 & 0 \\
0 & \frac{5}{4} & 0 & 0 \\
0 & 0 & -\frac{5}{4} & 0 \\
0 & 0 & 0 & -\frac{15}{4}
\end{pmatrix}.
\end{align*}
Moreover, the eigenvalue of the polynomial $Q_{d_1,d_2}$ is given by
$$\Gamma^{-}_{d_1,d_2}=\Gamma_{L,d_1,d_2}-\Gamma_{R,d_1,d_2}= \mathrm{diag} \begin{pmatrix}
\frac{3}{2} d_1 +d_2+ \frac{1}{2} d_3 +\frac{15}{4} \\
-\frac{1}{2} d_1 +d_2+ \frac{1}{2} d_3 +\frac{5}{4} \\
-\frac{1}{2} d_1 -d_2+ \frac{1}{2} d_3 -\frac{5}{4}\\
-\frac{1}{2} d_1 -d_2- \frac{3}{2} d_3 -\frac{15}{4}
\end{pmatrix}.$$

\end{document}